\documentclass[a4paper,twoside,10pt]{article}
\pdfoutput=1
\usepackage[utf8]{inputenc}
\usepackage[english]{babel}
\usepackage{hyperref}
\hypersetup{
colorlinks=true,
citecolor=black,
linkcolor=black,
urlcolor=black
}
\usepackage{geometry}
\usepackage{titlesec}
\titleformat*{\section}{\bfseries\normalsize}
\titleformat*{\subsection}{\itshape\normalsize}
\titleformat*{\subsubsection}{\itshape\normalsize}
\usepackage{amsmath,amssymb,amsthm,mathrsfs}
\theoremstyle{definition}\newtheorem{proposition}{Proposition}
\usepackage[algo2e,tworuled,vlined,linesnumbered,norelsize]{algorithm2e}
\usepackage{graphicx}
\usepackage[font={footnotesize}]{caption}
\usepackage[labelformat=parens,subrefformat=parens]{subcaption}
\usepackage{multirow,threeparttable,array}
\newcolumntype{L}[1]{>{\raggedright\let\newline\\\arraybackslash\hspace{0pt}}m{#1}}
\newcolumntype{C}[1]{>{\centering\let\newline\\\arraybackslash\hspace{0pt}}m{#1}}
\newcolumntype{R}[1]{>{\raggedleft\let\newline\\\arraybackslash\hspace{0pt}}m{#1}}
\usepackage{url}
\newcommand*{\rom}[1]{\romannumeral #1\relax}
\usepackage{cite}
\let\OLDthebibliography\thebibliography
\renewcommand\thebibliography[1]{
\OLDthebibliography{#1}
\setlength{\parskip}{0pt}
\setlength{\itemsep}{0pt}
\footnotesize
}
\usepackage{authblk}

\title{\Large Solution of the $k$-th eigenvalue problem in large-scale electronic structure calculations}
\author[a]{Dongjin Lee\thanks{\texttt{dongjin-i@na.nuap.nagoya-u.ac.jp}}}
\author[b]{Takeo Hoshi}
\author[a]{Tomohiro Sogabe}
\author[a]{Yuto Miyatake}
\author[a]{Shao-Liang Zhang}
\affil[a]{Graduate School of Engineering, Nagoya University, Furo-cho, Chikusa-ku, Nagoya 464-8603, Japan}
\affil[b]{Department of Applied Mathematics and Physics, Tottori University, 4-101 Koyama-Minami, Tottori 680-8550, Japan}
\date{}
\begin{document}
%---------%
\maketitle%
%---------%
\begin{abstract}
We consider computing the $k$-th eigenvalue and its corresponding eigenvector of a generalized Hermitian eigenvalue problem of $n\times n$ large sparse matrices.
In electronic structure calculations, several properties of materials, such as those of optoelectronic device materials, are governed by the eigenpair with a material-specific index $k.$
We present a three-stage algorithm for computing the $k$-th eigenpair with validation of its index.
In the first stage of the algorithm, we propose an efficient way of finding an interval containing the $k$-th eigenvalue $(1 \ll k \ll n)$ with a non-standard application of the Lanczos method.
In the second stage, spectral bisection for large-scale problems is realized using a sparse direct linear solver to narrow down the interval of the $k$-th eigenvalue.
In the third stage, we switch to a modified shift-and-invert Lanczos method to reduce bisection iterations and compute the $k$-th eigenpair with validation.
Numerical results with problem sizes up to 1.5 million are reported, and the results demonstrate the accuracy and efficiency of the three-stage algorithm.
\par
\end{abstract}
%--------------------------------------%
\section{Introduction}\label{sec:intro}%
%--------------------------------------%
In this paper, we consider the solution of the $k$-th eigenvalue problem with $n \times n$ large sparse Hermitian $A,B$ and positive definite $B$:
\begin{align}\label{eq:kep}
A \boldsymbol{x}_{k} = \lambda_{k} B \boldsymbol{x}_{k},\quad \boldsymbol{x}_{k} \neq \boldsymbol{0},
\end{align}
where $\lambda_{1} \le \cdots \le \lambda_{k} \le \cdots \le \lambda_{n}.$
Here, we assume that the problem-specific target index $k$ satisfies $1 \ll k \ll n$ such that $\lambda_{k}$ is not at either end of $[\lambda_{1},\lambda_{n}].$
\par
The $k$-th eigenvalue problem \eqref{eq:kep} differs from other problems to compute part of the spectrum and the corresponding eigenvectors.
Some eigenvalues at the ends of $[\lambda_{1},\lambda_{n}]$ and their corresponding eigenvectors can be computed by the Lanczos \cite{lanczos1950iteration} and LOBPCG \cite{knyazev2001toward} methods, and some eigenvalues near a given target point and their eigenvectors are computed by the shift-and-invert Lanczos (SI Lanczos) \cite{ericsson1980spectral} and Jacobi--Davidson \cite{sleijpen1996jacobi} methods.
In addition, eigenvalues in a given target interval and their eigenvectors are computed by the Sakurai--Sugiura method \cite{sakurai2003projection}, FEAST method \cite{PhysRevB.79.115112}, and filtering methods \cite{li2016thick}.
However, none of these methods aim at computing the eigenpair of a given target index $k$ with $1 \ll k \ll n.$
\par
The $k$-th eigenvalue problem \eqref{eq:kep} arises from large-scale electronic structure calculations \cite{hoshi2012order}, where eigenvalues correspond to the energy of an electron and eigenvectors represent an electronic wave function.
Here, $\lambda_{k}$ and $\boldsymbol{x}_{k}$ are referred to as the highest occupied (HO) energy and state, respectively.
The target index $k$ is a material-specific value and is approximately 10--50\% of the matrix size \cite{marek2014elpa}, thereby satisfying $1 \ll k \ll n.$
The index of the eigenpair must be validated because several of the physical properties of materials, such as those of optoelectronic device materials, are governed by the eigenpair with the material-specific target index $k$.
Detailed explanation of physical origin and background of the $k$-th eigenvalue problem can be found in \hyperref[sec:appendix]{Appendix}.
\par
Typical electronic structure calculations require the computation of many eigenpairs.
Eigenpairs of practical interest are generally $(\lambda_{1},\boldsymbol{x}_{1})$ through $(\lambda_{k},\boldsymbol{x}_{k}).$
A standard approach is to utilize a dense eigensolver in a massively parallel environment \cite{slug1997,imamura2011development,marek2014elpa}.
Recently, a one-million-dimensional generalized eigenvalue problem was solved by a dense eigensolver \cite{imachi2016hybrid} using the full K computer.
The elapsed time was 5516 seconds to compute all eigenpairs \cite{hoshi2016extremely}, indicating the practical size limit of eigenpair computation by a dense eigensolver.
Therefore, a strong need for methodologies that can be applied to larger materials and matrices has become apparent.
\par
A promising approach to large-scale electronic structure calculations is to construct purpose-specific methods with which one can bypass computation of many eigenpairs to obtain several physical quantities of practical interest.
Several methods have been proposed to calculate the total energy of materials and the force on nuclei in order to realize quantum molecular dynamics simulations without computation of each individual eigenpair \cite{bowler2012methods}.
Such methods are referred to as linear-scaling methods.
\par
Our previous paper \cite{lee2013interior} presented the $k$-th eigenvalue problem \eqref{eq:kep} as a purpose-specific methodology for large-scale calculations that can be considered complementary to linear-scaling methods.
In our previous paper, as a preliminary study of the problem, spectral bisection and its variants were applied to computing the $k$-th eigenvalue and validating its index.
As explained in Section \ref{sec:bp:class:bisect}, the main idea behind bisection is to locate $\lambda_{k}$ based on the number of eigenvalues that are less than a given real number $\sigma,$ denoted $\nu_{\sigma}(A,B).$
\par
Using our previous paper as a foundation, this paper presents a three-stage algorithm for solving the $k$-th eigenvalue problem with the following features.
\begin{enumerate}
\item Efficient initial interval for bisection\\
Setting an initial interval containing $\lambda_{k}$ is necessary to begin bisection.
A standard approach is to utilize some Gershgorin-type theorem so that the interval includes the entire spectrum and thus contains $\lambda_{k}.$
In the first stage of the algorithm, we propose an efficient way of finding a narrow initial interval.
The proposed approach iteratively generates a sequence of disjoint intervals until an interval validated as containing $\lambda_{k}$ is obtained.
For our problem with the target index $1 \ll k \ll n,$ the numerical results show that information conventionally considered useless in the Lanczos method, i.e., Ritz values of the first few steps of the method, is of paramount importance to generate a narrow initial interval.
\item Computation of the $k$-th eigenvector with validation\\
Mistakenly computing another eigenvector, e.g., $\boldsymbol{x}_{k-1}$ or $\boldsymbol{x}_{k+1},$ rather than $\boldsymbol{x}_{k}$ leads to a completely unreliable result because eigenvectors are $B$-orthogonal to each other.
In the third stage, as will be shown in Proposition \ref{prop:bound}, the index is validated by utilizing an eigenvalue error bound that can be evaluated from the residual vector at negligible cost.
\item Application to large-scale problems\\
Each bisection iteration requires computation of $\nu_{\sigma}(A,B),$ which is based on $LDL^{\mathrm{H}}$ factorization of a shifted matrix $A - \sigma B$ and thus requires $O(n^{2})$ memory in general.
In the second stage, $\nu_{\sigma}(A,B)$ is computed by utilizing a sparse direct linear solver to save memory to realize bisection for large-scale problems.
Once a fill-reducing ordering and symbolic factorization are obtained for some $A -\sigma B,$ they can be recycled for other shifted matrices because they depend on only the sparsity structure of a matrix.
\end{enumerate}
\par
The remainder of this paper is organized as follows.
Section \ref{sec:bp} explains the preliminaries of this paper.
In Section \ref{sec:app}, after explaining our approach to set an initial interval and compute the $k$-th eigenvector, we present the three-stage algorithm for the $k$-th eigenvalue problem.
Numerical results of several real research problems and a comparison of the three-stage algorithm and dense eigensolvers are reported in Section \ref{sec:numerical}.
Concluding remarks are given in Section \ref{sec:conclusion}.
\par
Throughout this paper, $A^{\mathrm{T}}$ and $A^{\mathrm{H}}$ denote the transpose and conjugate transpose of matrix $A,$ respectively.
$I$ denotes the identity matrix.
For Hermitian positive definite $B,\ \| \boldsymbol{x} \|_{B}$ denotes the $B$-norm of vector $\boldsymbol{x},$ defined as $\| \boldsymbol{x} \|_{B} = \sqrt{\boldsymbol{x}^{\mathrm{H}} B \boldsymbol{x}}.$
\par
%------------------------------------%
\section{Preliminaries}\label{sec:bp}%
%------------------------------------%
Section \ref{sec:bp:class:bisect} explains spectral bisection for computing the $k$-th eigenvalue.
Sections \ref{sec:bp:pre:lanczos} and \ref{sec:bp:pre:silanczos} provide the preliminaries of the Lanczos and SI Lanczos methods.
\par
%-----------------------------------------------------------------------------------%
\subsection{Spectral bisection for the $k$-th eigenvalue}\label{sec:bp:class:bisect}%
%-----------------------------------------------------------------------------------%
In spectral bisection, $\nu_{\sigma}(A,B)$ is computed to determine whether $\lambda_{k}$ is located to the left or right of $\sigma.$
Specifically, $k \le \nu_{\sigma}(A,B)$ implies that $\lambda_{k} < \sigma,$ while $k > \nu_{\sigma}(A,B)$ implies that $\lambda_{k} \ge \sigma.$
This idea dates back to the work of Givens \cite{givens1954numerical}.
\par
Algorithm \ref{alg:bisect} shows the spectral bisection for the $k$-th eigenvalue, which was employed but not clearly presented in our previous paper \cite{lee2013interior}.
When an interval $[\sigma_{\mathrm{lower}},\sigma_{\mathrm{upper}})$ containing $\lambda_{k}$ is set, the midpoint $\sigma$ of the interval is calculated in line \ref{alg:bisect:root}, and then
$\nu_{\sigma}(A,B)$ is computed by utilizing $LDL^{\mathrm{H}}$ factorization of the shifted matrix $A - \sigma B$ in lines \ref{alg:bisect:ldl}--\ref{alg:bisect:diag}.
Here, $A - \sigma B$ is indefinite; thus, permutation $P$ is necessary for numerically stable factorization \cite{BunchKaufman1977}, and $L$ and $D$ are a unit lower triangular matrix and a block diagonal matrix of block size one or two, respectively.
Due to Sylvester's law of inertia, $\nu_{0}(D,I)$ equals $\nu_{\sigma}(A,B).$
Depending on $\nu_{\sigma}(A,B),$ either the left or right half-interval is selected as the next interval in line \ref{alg:bisect:half}.
$\lceil \log_{2}[(\sigma_{\mathrm{upper}} - \sigma_{\mathrm{lower}})/\tau] \rceil$ iterations are required until the interval becomes narrower than a given tolerance $\tau.$
Thus, a narrower initial interval will result in fewer required iterations to locate $\lambda_{k}.$
\par
\begin{algorithm2e}[H]
\SetNlSty{textrm}{}{}\SetNlSkip{0.625em}\DontPrintSemicolon
\SetKwInOut{Input}{Input}\SetKwInOut{Output}{Output}
\SetKwIF{If}{ElseIf}{Else}{if}{then}{else if}{else}{{end if}}
\SetKwIF{RIf}{RElseIf}{RElse}{repeat until}{}{}{}{{}}
\SetKwComment{tcp}{$\triangleright$~}{}\SetCommentSty{textrm}
\Input{matrices $A,B$ of generalized eigenvalue problem \eqref{eq:kep}, target index $k,$\\
initial interval $[\sigma_{\mathrm{lower}},\sigma_{\mathrm{upper}})$ containing the $k$-th eigenvalue, tolerance $\tau.$}
\Output{approximate eigenvalue $\hat{\lambda}_{k}:=(\sigma_{\mathrm{lower}}+\sigma_{\mathrm{upper}})/2,$ where $|\hat{\lambda}_{k} - \lambda_{k}| < \tau/2.$}
\RIf{$\sigma_{\mathrm{upper}} - \sigma_{\mathrm{lower}} < \tau$}{
	$\sigma:=(\sigma_{\mathrm{lower}}+\sigma_{\mathrm{upper}})/2,$\label{alg:bisect:root}\;
	$LDL^{\mathrm{H}}\leftarrow P(A-\sigma B)P^{\mathrm{T}},$\label{alg:bisect:ldl}\tcp*{\makebox[27em]{$P:$ permutation for numerical stability\hfill}}
	$\nu:=\nu_{0}(D,I),\label{alg:bisect:diag}$\tcp*{\makebox[27em]{$\nu_{0}(D,I):$ number of negative eigenvalues of block diagonal $D$\hfill}}
	\leIf{$k \le \nu$}{$\sigma_{\mathrm{upper}}:=\sigma$}{$\sigma_{\mathrm{lower}}:=\sigma.$}\label{alg:bisect:half}
}
\caption{Spectral bisection for the $k$-th eigenvalue \cite{lee2013interior}\label{alg:bisect}}
\end{algorithm2e}
\par
Rather than bisection, it is possible to apply other root-finding algorithms to line \ref{alg:bisect:root} of Algorithm \ref{alg:bisect} to achieve faster convergence.
By considering eigenvalues as the roots of a characteristic polynomial $\det(A - \sigma B)$, several root-finding algorithms are applied to select point $\sigma$ to compute $\nu_{\sigma}(A,B)$ \cite{peters1969eigenvalues}.
Generally, such variants can be applied only after the interval is narrowed down sufficiently to contain only one eigenvalue \cite[Chapter 3.5]{parlett1998symmetric}.
Our previous paper \cite{lee2013interior} took another perspective and considered $\nu_{\sigma}(A,B)$ as a function of $\sigma,$ which is non-decreasing, integer-valued, and discontinuous at $\sigma = \lambda_{i}$ for $1 \le i \le n.$
This perspective enabled the application of a certain type of root-finding algorithms to select $\sigma$ before $\lambda_{k}$ was isolated from the other eigenvalues by the interval.
To date, a thorough comparison of bisection and its variants has not been made for the $k$-th eigenvalue problem, and, in this paper, we utilize bisection for the sake of stable performance.
\par
%----------------------------------------------------%
\subsection{Lanczos method}\label{sec:bp:pre:lanczos}%
%----------------------------------------------------%
The Lanczos method \cite{lanczos1950iteration} is a projection method in which approximate solutions are constructed within a Krylov subspace and are determined to be optimal in the sense of the Galerkin condition.
The subspace and its orthonormal basis are generated by the Lanczos process, which can be expressed in the matrix form:
\begin{align}\label{eq:lanczos}
A V_{j} = B V_{j} T_{j} + B \boldsymbol{v}_{j+1} \beta_{j} \boldsymbol{e}_{j}^{\mathrm{T}}.
\end{align}
Here, $V_{j}$ is an $n \times j$ matrix whose columns span a Krylov subspace and are $B$-orthonormal.
$T_{j}$ is real symmetric tridiagonal and has non-zero off-diagonal elements (irreducible).
$\boldsymbol{v}_{j+1}$ is $B$-orthogonal to the columns of $V_{j}$ and is normalized with respect to the $B$-norm by scale factor $\beta_{j}$.
$\boldsymbol{e}_{j}$ is the last column of the identity matrix of size $j.$
In this paper, we refer to \eqref{eq:lanczos} as $j$-step Lanczos decomposition.
\par
From the Galerkin condition, the standard eigenvalue problem of $T_{j}$ is derived:
\begin{align}\label{eq:triep}%
T_{j} \boldsymbol{y}_{i}^{(j)} = \theta_{i}^{(j)} \boldsymbol{y}_{i}^{(j)},\quad \boldsymbol{y}_{i}^{(j)} \neq \boldsymbol{0}.
\end{align}%
Since $T_{j}$ is irreducible, eigenvalues $\theta_{i}^{(j)}$ are distinct from each other \cite[Chapter 1.3]{meurant2006lanczos} and can be indexed in increasing order, i.e., $\theta_{1}^{(j)} < \theta_{2}^{(j)} < \cdots < \theta_{j}^{(j)}.$
The Lanczos method can be considered a Rayleigh--Ritz procedure, and eigenvalues $\theta_{i}^{(j)}$ are referred to as Ritz values.
\par
%-----------------------------------------------------------------------%
\subsection{Shift-and-invert Lanczos method}\label{sec:bp:pre:silanczos}%
%-----------------------------------------------------------------------%
A small number of eigenvalues near a target point (or in a target interval) and their associated eigenvectors can be computed by the SI Lanczos method \cite{ericsson1980spectral}.
The SI Lanczos method applied to the original problem \eqref{eq:kep} can be considered as applying the original Lanczos method (Section \ref{sec:bp:pre:lanczos}) to the SI problem:
\begin{align}\label{eq:sigep}
(A - \sigma B)^{-1} \tilde{\boldsymbol{x}} = \tilde{\lambda} B^{-1} \tilde{\boldsymbol{x}},\quad \tilde{\boldsymbol{x}} \neq \boldsymbol{0}.
\end{align}
Here, it is assumed that shift $\sigma$ does not coincide with an eigenvalue of \eqref{eq:kep}, i.e., $\sigma \neq \lambda_{i}.$
The eigenvalues of the original and SI problems have the relationship $\tilde{\lambda} = (\lambda - \sigma)^{-1},$ while eigenvectors satisfy both $\tilde{\boldsymbol{x}} = B\boldsymbol{x}$ and $\tilde{\boldsymbol{x}} = (A - \sigma B) \boldsymbol{x}$ because $B\boldsymbol{x}$ and $(A - \sigma B) \boldsymbol{x}$ are collinear.
Based on these relationships, approximate eigenpairs for the SI problem \eqref{eq:sigep} are transformed to those for the original problem \eqref{eq:kep}.
\par
The matrix form of $j$-step SI Lanczos decomposition is given as follows:
\begin{align}\label{eq:sigep:lanczos}
(A - \sigma B)^{-1} \tilde{V}_{j} = B^{-1} \tilde{V}_{j} \tilde{T}_{j} + B^{-1} \tilde{\boldsymbol{v}}_{j+1} \tilde{\beta}_{j} \boldsymbol{e}_{j}^{\mathrm{T}}.
\end{align}
Here, $\tilde{V}_{j}$ is an $n \times j$ matrix whose columns span a Krylov subspace and are $B^{-1}$-orthonormal, and $\tilde{T}_{j}$ is real symmetric tridiagonal and irreducible.
$\tilde{\boldsymbol{v}}_{j+1}$ is $B^{-1}$-orthogonal to the columns of $\tilde{V}_{j}$ and is normalized with respect to the $B^{-1}$-norm by scale factor $\tilde{\beta}_{j}$.
\par
From the Galerkin condition, the standard eigenvalue problem of $\tilde{T}_{j}$ is derived:
\begin{align}\label{eq:sigep:triep}
\tilde{T}_{j} \tilde{\boldsymbol{y}}_{i}^{(j)} = \tilde{\theta}_{i}^{(j)} \tilde{\boldsymbol{y}}_{i}^{(j)},\quad \tilde{\boldsymbol{y}}_{i}^{(j)} \neq \boldsymbol{0}.
\end{align}
Here, eigenvalues $\tilde{\theta}_{i}^{(j)}$ are indexed in increasing order, and, in the remainder of this paper, eigenvectors $\tilde{\boldsymbol{y}}_{i}^{(j)}$ are assumed to be normalized with respect to the $2$-norm.
Using the $i$-th eigenpair $(\tilde{\theta}_{i}^{(j)},\tilde{\boldsymbol{y}}_{i}^{(j)})$ of $\tilde{T}_{j},$ the approximate eigenvalues for \eqref{eq:kep} are expressed as follows:
\begin{align}\label{eq:sigep:value}
\lambda_{i}^{(j)} = \sigma+1/\tilde{\theta}_{i}^{(j)}.
\end{align}
Approximate eigenvectors can be expressed in two different ways.
\begin{enumerate}
\item Using the first relationship $\tilde{\boldsymbol{x}} = B\boldsymbol{x},$ approximate eigenvectors are given as:
\begin{align}\label{eq:sigep:vec1}
\boldsymbol{x}_{i,1}^{(j)} = B^{-1} \tilde{V}_{j} \tilde{\boldsymbol{y}}_{i}^{(j)}.
\end{align}
\item From the second relationship $\tilde{\boldsymbol{x}} = (A - \sigma B) \boldsymbol{x},$ we have:
\begin{align}\label{eq:sigep:vec2}
\boldsymbol{x}_{i,2}^{(j)} = (A - \sigma B)^{-1} \tilde{V}_{j} \tilde{\boldsymbol{y}}_{i}^{(j)}.
\end{align}
\end{enumerate}
\par
In this paper, we utilize $\boldsymbol{x}_{i,2}^{(j)}$ in \eqref{eq:sigep:vec2} as an approximate eigenvector, although it is common to use $\boldsymbol{x}_{i,1}^{(j)}$ in \eqref{eq:sigep:vec1}.
This is because, as will be shown in Proposition \ref{prop:bound}, $\boldsymbol{x}_{i,2}^{(j)}$ allows economical evaluation of an eigenvalue error bound that can be utilized to validate the index of approximate eigenpairs.
Other perspectives on approximate eigenvectors and their further treatment can be found in the literature \cite{ericsson1980spectral} and \cite[Chapter 7.6.8]{bai2000templates}, in which $\boldsymbol{x}_{i,2}^{(j)}$ is considered a modification of $\boldsymbol{x}_{i,1}^{(j)}.$
\par
The remainder of this subsection provides theoretical results that support the utilization of $\boldsymbol{x}_{i,2}^{(j)}$ as an approximate eigenvector.
In Proposition \ref{prop:ev}, we show that $\boldsymbol{x}_{i,2}^{(j)}$ converges to the same eigenvector of \eqref{eq:kep} at the same iteration of the SI Lanczos method as $\boldsymbol{x}_{i,1}^{(j)}.$
Then, in Proposition \ref{prop:ortho}, we show that $\boldsymbol{x}_{i,2}^{(j)}$ with $1 \le i \le j$ become $B$-orthogonal to each other as the SI Lanczos method proceeds.
\par
\begin{proposition}\label{prop:ev}
Approximate eigenvectors $\boldsymbol{x}_{i,1}^{(j)}$ in \eqref{eq:sigep:vec1} and $\boldsymbol{x}_{i,2}^{(j)}$ in \eqref{eq:sigep:vec2} are collinear if and only if $\tilde{\boldsymbol{v}}_{j+1}=\boldsymbol{0}.$
\end{proposition}
\begin{proof}%
We first prove the sufficiency.
By post-multiplying \eqref{eq:sigep:lanczos} by $\tilde{\boldsymbol{y}}_{i}^{(j)}$ and from the sufficient condition $\tilde{\boldsymbol{v}}_{j+1}=\boldsymbol{0},$ we have:
\begin{align*}%
\boldsymbol{x}_{i,2}^{(j)} = (A - \sigma B)^{-1} \tilde{V}_{j} \tilde{\boldsymbol{y}}_{i}^{(j)} = B^{-1} \tilde{V}_{j} \tilde{T}_{j} \tilde{\boldsymbol{y}}_{i}^{(j)} = \tilde{\theta}_{i}^{(j)} \boldsymbol{x}_{i,1}^{(j)}.
\end{align*}
The third equality follows from \eqref{eq:sigep:triep}.
This proves the sufficiency.
Now, we prove the necessity.
By post-multiplying \eqref{eq:sigep:lanczos} by $\tilde{\boldsymbol{y}}_{i}^{(j)},$ we have:
\begin{align*}
\boldsymbol{x}_{i,2}^{(j)} = \tilde{\theta}_{i}^{(j)} \boldsymbol{x}_{i,1}^{(j)} + B^{-1} \tilde{\boldsymbol{v}}_{j+1} (\tilde{\beta}_{j} \boldsymbol{e}_{j}^{\mathrm{T}} \tilde{\boldsymbol{y}}_{i}^{(j)}).
\end{align*}
Here, $\boldsymbol{x}_{i,1}^{(j)}$ and $\boldsymbol{x}_{i,2}^{(j)}$ are collinear from the necessary condition; thus, the last term $B^{-1} \tilde{\boldsymbol{v}}_{j+1} (\tilde{\beta}_{j} \boldsymbol{e}_{j}^{\mathrm{T}} \tilde{\boldsymbol{y}}_{i}^{(j)})$ must be collinear with $\boldsymbol{x}_{i,1}^{(j)}$ and $\boldsymbol{x}_{i,2}^{(j)}.$
In addition, the last term is $B$-orthogonal to $\boldsymbol{x}_{i,1}^{(j)}$ because $\tilde{\boldsymbol{v}}_{j+1}$ is $B^{-1}$-orthogonal to the columns of $\tilde{V}_{j}$:
\begin{align*}
(\boldsymbol{x}_{i,1}^{(j)})^{\mathrm{H}} B \left[B^{-1} \tilde{\boldsymbol{v}}_{j+1} (\tilde{\beta}_{j} \boldsymbol{e}_{j}^{\mathrm{T}} \tilde{\boldsymbol{y}}_{i}^{(j)})\right] = (\tilde{\boldsymbol{y}}_{i}^{(j)})^{\mathrm{H}} \tilde{V}_{j}^{\mathrm{H}} B^{-1} \tilde{\boldsymbol{v}}_{j+1} (\tilde{\beta}_{j} \boldsymbol{e}_{j}^{\mathrm{T}} \tilde{\boldsymbol{y}}_{i}^{(j)}) = 0.
\end{align*}
Due to this $B$-orthogonality and the positive-definiteness of $B,$ the last term can never be collinear with $\boldsymbol{x}_{i,1}^{(j)}$ unless it is the zero vector.
Therefore, $\tilde{\boldsymbol{v}}_{j+1} = \boldsymbol{0}$ (thus, $\tilde{\beta}_{j}=0$) or $\boldsymbol{e}_{j}^{\mathrm{T}} \tilde{\boldsymbol{y}}_{i}^{(j)} = 0.$
However, $\boldsymbol{e}_{j}^{\mathrm{T}} \tilde{\boldsymbol{y}}_{i}^{(j)}$ is non-zero because it is the last element of an eigenvector of an irreducible tridiagonal matrix \cite[Theorem 7.9.3]{parlett1998symmetric}.
This proves the necessity.
\end{proof}
\par
Note that $\boldsymbol{x}_{i,2}^{(j)}$ with $1 \le i \le j$ do not have exact $B$-orthogonality.
Their $B$-orthogonality can be measured by \eqref{eq:sigep:orthonorm} in Proposition \ref{prop:ortho}, which is the cosine similarity of two approximate eigenvectors in the $B$-inner product.
\par
\begin{proposition}\label{prop:ortho}
The following holds for $1 \le l < m \le j$:
\begin{align}\label{eq:sigep:orthonorm}
\frac{| (\boldsymbol{x}_{l,2}^{(j)})^{\mathrm{H}} B \boldsymbol{x}_{m,2}^{(j)} |}{\| \boldsymbol{x}_{l,2}^{(j)} \|_{B} \cdot \| \boldsymbol{x}_{m,2}^{(j)}\|_{B}} = \frac{|\tilde{\beta}_{j}\boldsymbol{e}_{j}^{\mathrm{T}}\tilde{\boldsymbol{y}}_{l}^{(j)}/\tilde{\theta}_{l}^{(j)}|}{\sqrt{1 + |\tilde{\beta}_{j}\boldsymbol{e}_{j}^{\mathrm{T}}\tilde{\boldsymbol{y}}_{l}^{(j)}/\tilde{\theta}_{l}^{(j)}|^{2}}} \cdot \frac{|\tilde{\beta}_{j}\boldsymbol{e}_{j}^{\mathrm{T}}\tilde{\boldsymbol{y}}_{m}^{(j)}/\tilde{\theta}_{m}^{(j)}|}{\sqrt{1 + |\tilde{\beta}_{j}\boldsymbol{e}_{j}^{\mathrm{T}}\tilde{\boldsymbol{y}}_{m}^{(j)}/\tilde{\theta}_{m}^{(j)}|^{2}}}.
\end{align}
\end{proposition}
\begin{proof}
For $1 \le l \le m \le j,$
\begin{align}\label{eq:sigep:ortho}
(\boldsymbol{x}_{l,2}^{(j)})^{\mathrm{H}} B \boldsymbol{x}_{m,2}^{(j)} &= \left[(A -\sigma B)^{-1} \tilde{V}_{j} \tilde{\boldsymbol{y}}_{l}^{(j)}\right]^{\mathrm{H}} B \left[(A -\sigma B)^{-1} \tilde{V}_{j} \tilde{\boldsymbol{y}}_{m}^{(j)}\right]\nonumber\\
&= \left(\tilde{V}_{j} \tilde{T}_{j}\tilde{\boldsymbol{y}}_{l}^{(j)} + \tilde{\boldsymbol{v}}_{j+1}\tilde{\beta}_{j}\boldsymbol{e}_{j}^{\mathrm{T}}\tilde{\boldsymbol{y}}_{l}^{(j)}\right)^{\mathrm{H}} B^{-1} B B^{-1} \left(\tilde{V}_{j} \tilde{T}_{j}\tilde{\boldsymbol{y}}_{m}^{(j)} + \tilde{\boldsymbol{v}}_{j+1}\tilde{\beta}_{j}\boldsymbol{e}_{j}^{\mathrm{T}}\tilde{\boldsymbol{y}}_{m}^{(j)}\right)\nonumber\\
&= \left(\tilde{V}_{j} \tilde{T}_{j}\tilde{\boldsymbol{y}}_{l}^{(j)}\right)^{\mathrm{H}} B^{-1} \left(\tilde{V}_{j} \tilde{T}_{j}\tilde{\boldsymbol{y}}_{m}^{(j)}\right) + \left(\tilde{\boldsymbol{v}}_{j+1}\tilde{\beta}_{j}\boldsymbol{e}_{j}^{\mathrm{T}}\tilde{\boldsymbol{y}}_{l}^{(j)}\right)^{\mathrm{H}} B^{-1} \left(\tilde{\boldsymbol{v}}_{j+1}\tilde{\beta}_{j}\boldsymbol{e}_{j}^{\mathrm{T}}\tilde{\boldsymbol{y}}_{m}^{(j)}\right)\nonumber\\
&= \tilde{\theta}_{l}^{(j)}\tilde{\theta}_{m}^{(j)} \left[ \delta_{lm} + (\tilde{\beta}_{j}\boldsymbol{e}_{j}^{\mathrm{T}}\tilde{\boldsymbol{y}}_{l}^{(j)}/\tilde{\theta}_{l}^{(j)})(\tilde{\beta}_{j}\boldsymbol{e}_{j}^{\mathrm{T}}\tilde{\boldsymbol{y}}_{m}^{(j)}/\tilde{\theta}_{m}^{(j)}) \right].
\end{align}
Here, $\delta_{lm}$ denotes the Kronecker delta.
The second equality follows from \eqref{eq:sigep:lanczos}, and the third and fourth are due to the $B^{-1}$-orthonormality of the columns of $\tilde{V}_{j}$ and $\tilde{\boldsymbol{v}}_{j+1}.$
\eqref{eq:sigep:orthonorm} is an immediate result of \eqref{eq:sigep:ortho} with $l \neq m.$
\end{proof}
\par
Scalars $|\tilde{\beta}_{j}\boldsymbol{e}_{j}^{\mathrm{T}}\tilde{\boldsymbol{y}}_{i}^{(j)}/\tilde{\theta}_{i}^{(j)}|$ in \eqref{eq:sigep:orthonorm} are simply the $B^{-1}$ norm of the residual vectors:
\begin{align}\label{eq:sigep:res2}
\boldsymbol{r}_{i,2}^{(j)} \equiv (A - \lambda_{i}^{(j)} B) \boldsymbol{x}_{i,2}^{(j)} = \left[A - (\sigma + 1/\tilde{\theta}_{i}^{(j)} ) B\right] (A - \sigma B)^{-1} \tilde{V}_{j} \tilde{\boldsymbol{y}}_{i}^{(j)} = - \tilde{\boldsymbol{v}}_{j+1} (\tilde{\beta}_{j}\boldsymbol{e}_{j}^{\mathrm{T}}\tilde{\boldsymbol{y}}_{i}^{(j)}/\tilde{\theta}_{i}^{(j)}).
\end{align}
Since the residual norm $|\tilde{\beta}_{j}\boldsymbol{e}_{j}^{\mathrm{T}}\tilde{\boldsymbol{y}}_{i}^{(j)}/\tilde{\theta}_{i}^{(j)}| \ll 1$ as the SI Lanczos method proceeds, \eqref{eq:sigep:orthonorm} converges to zero and approximate eigenvectors $\boldsymbol{x}_{i,2}^{(j)}$ become $B$-orthogonal to each other.
Therefore, we can say that $\boldsymbol{x}_{i,2}^{(j)}$ have $B$-orthogonality in a practical sense.
\par
%---------------------------------------------------------------------------------%
\section{A three-stage algorithm for the $k$-th eigenvalue problem}\label{sec:app}%
%---------------------------------------------------------------------------------%
%-----------------------------------------------------------------------%
\subsection{Efficient initial interval for bisection}\label{sec:app:eff}%
%-----------------------------------------------------------------------%
To utilize spectral bisection, it is necessary to set an initial interval that contains the $k$-th eigenvalue.
A common approach is to set an interval that includes the entire spectrum, which necessarily contains $\lambda_{k}.$
For standard eigenvalue problems, one of the most economical ways to set such an interval is to use the Gershgorin circle theorem.
Several Gershgorin-type theorems have been proposed for generalized eigenvalue problems \cite{Stewart1975,KosticCvetkovicVarga2009,Nakatsukasa2011}.
These theorems provide an inclusion set of the spectrum that is guaranteed to be bounded only when at least one of $A$ and $B$ is diagonally dominant for each row \cite{Nakatsukasa2011}.
Unfortunately, such diagonal dominance of the matrices is not always assumed to be the case for problems in electronic structure calculations.
The remainder of this subsection presents a systematic way to set an initial interval based on an application of the Lanczos method.
\par
Ritz values by the Lanczos method become more accurate by expanding the subspace, exhibiting the interlacing property \cite[Chapter 1.3]{meurant2006lanczos}.
By denoting the $i$-th Ritz value at the $j$-th iteration of the method as $\theta^{(j)}_{i}$ (Section \ref{sec:bp:pre:lanczos}), the interlacing property can be formally expressed as follows: $\theta^{(j + 1)}_{i} < \theta^{(j)}_{i} < \theta^{(j+1)}_{i + 1}$ for $i\le j < n.$
The monotonic convergence follows from this property, which states that the $i$-th smallest (resp. largest) Ritz value decreases (resp. increases) monotonically and converges to the $i$-th smallest (resp. largest) eigenvalue.
We utilize this monotonicity of Ritz values to set an initial interval.
\par
Algorithm \ref{alg:initial} shows how we set an interval containing $\lambda_{k}$ by utilizing Ritz values, and this process is illustrated in Figure \ref{fig:initial}.
Here, we utilize either the smallest $\theta^{(j)}_{1}$ or largest Ritz values $\theta^{(j)}_{j}$ with $j \ge 1.$
We begin by computing $\theta_{1}^{(1)},$ which is equal to the generalized Rayleigh quotient $\boldsymbol{v}_{1}^{\mathrm{H}} A \boldsymbol{v}_{1} / \boldsymbol{v}_{1}^{\mathrm{H}} B \boldsymbol{v}_{1}$ of a random starting vector $\boldsymbol{v}_{1}.$
The quotient is expected to be the average of the eigenvalues of \eqref{eq:kep}.
It is advantageous to begin from approximately the middle of the eigenvalue distribution when the target index satisfies $1 \ll k \ll n.$
We then compute $\nu^{(1)} = \nu_{\theta_{1}^{(1)}}(A,B).$
If $k \le \nu^{(1)}$, it follows that $\lambda_{k} < \theta_{1}^{(1)} < \theta_{j}^{(j)}$ for $j > 1.$
Therefore, in the subsequent iterations, the smallest Ritz values $\theta_{1}^{(j)}$ are utilized as the points for setting an interval containing $\lambda_{k}.$
On the other hand, if $k > \nu^{(1)}$, the largest Ritz values are selected as the endpoints of an interval.
Accordingly, a sequence of disjoint intervals, i.e., either $[\theta^{(j)}_{1},\theta^{(j-1)}_{1})$ or $[\theta^{(j-1)}_{j-1},\theta^{(j)}_{j})$ with $j > 1,$ is generated until an interval validated as containing $\lambda_{k}$ is obtained.
\par
An advantage of utilizing Ritz values is that an initial interval is necessarily included in and can be much narrower than $[\lambda_{1},\lambda_{n}],$ which leads to a reduction of the number of bisection iterations.
However, we must consider the cost of setting the interval, i.e., $j$ iterations of the Lanczos method and $j$ $LDL^{\mathrm{H}}$ factorizations.
The interlacing property provides a theoretical upper bound of the iteration count required to set an interval.
Here, since $\theta_{1}^{(n-k+1)} < \lambda_{k} < \theta_{k}^{(k)},$ at most ${n-k+1}$ iterations are required for the $k \le \nu^{(1)}$ case, and $k$ iterations are required for the $k > \nu^{(1)}$ case.
In practice, $j$ can be very small because eigenvalues at the ends of $[\lambda_{1},\lambda_{n}]$ are rapidly approximated by the Lanczos method and the target index $k$ is approximately 10--50\% of the matrix size $n,$ thereby satisfying $1 \ll k \ll n.$
In the numerical results presented in Section \ref{sec:numerical:overview:initial}, the iteration count $j$ is two for all experiments.
\par
\begin{algorithm2e}[H]
\SetNlSty{textrm}{}{}\SetNlSkip{0.625em}\DontPrintSemicolon
\SetKwInOut{Input}{Input}\SetKwInOut{Output}{Output}
\SetKwFor{For}{for}{do}{{end for}}
\SetKwIF{If}{ElseIf}{Else}{if}{then}{else if}{else}{{end if}}
\SetKwComment{tcp}{$\triangleright$~}{}\SetCommentSty{textrm}
\Input{matrices $A,B$ of generalized eigenvalue problem \eqref{eq:kep}, target index $k.$}
\Output{interval $[\sigma_{\mathrm{lower}},\sigma_{\mathrm{upper}})$ containing the $k$-th eigenvalue,\\
$\nu_{\mathrm{lower}}=\nu_{\sigma_{\mathrm{lower}}}(A,B),\ \nu_{\mathrm{upper}}=\nu_{\sigma_{\mathrm{upper}}}(A,B).$}
set a random starting vector $\boldsymbol{v}_{1},\ \boldsymbol{v}_{1} := \boldsymbol{v}_{1} / \| \boldsymbol{v}_{1} \|_{B},\ \nu^{(0)}:=0,$\;
\For{$j=1,2,\ldots$}{
	compute $j$-step Lanczos decomposition \eqref{eq:lanczos},\;
	solve standard eigenvalue problem \eqref{eq:triep},\;
	\leIf{$k \le \nu^{(j-1)}$}{$\sigma^{(j)}:=\theta^{(j)}_{1}$}{$\sigma^{(j)}:=\theta^{(j)}_{j},$}
	$LDL^{\mathrm{H}}\leftarrow P(A-\sigma^{(j)}B)P^{\mathrm{T}},$\tcp*{\makebox[27em]{$P:$ permutation for numerical stability\hfill}}\label{alg:initial:ordering}
	$\nu^{(j)}:=\nu_{0}(D,I),$\tcp*{\makebox[27em]{$\nu_{0}(D,I):$ number of negative eigenvalues of block diagonal $D$\hfill}}
	\lIf{\upshape $j \neq 1$ and $(\nu^{(j)} < k \le \nu^{(j-1)}$ or $\nu^{(j-1)} < k \le \nu^{(j)})$}{break,}
}
$\sigma_{\mathrm{lower}}:=\min\{\sigma^{(j-1)},\sigma^{(j)}\},\ \sigma_{\mathrm{upper}}:=\max\{\sigma^{(j-1)},\sigma^{(j)}\},$\;
$\nu_{\mathrm{lower}}:=\min\{\nu^{(j-1)},\nu^{(j)}\},\ \nu_{\mathrm{upper}}:=\max\{\nu^{(j-1)},\nu^{(j)}\}.$\;
\caption{Setting an interval containing the $k$-th eigenvalue\label{alg:initial}}
\end{algorithm2e}
\par
\begin{figure}[htbp]
\centering
\includegraphics[width=0.8\linewidth]{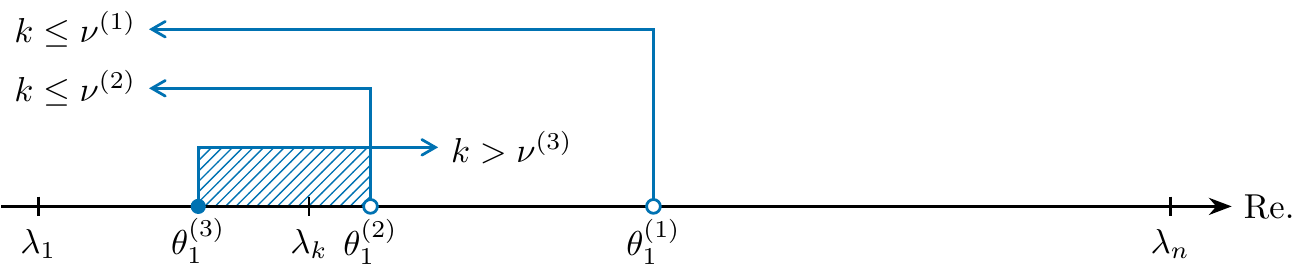}
\caption{Illustration of Algorithm \ref{alg:initial}.
$[\theta_{1}^{(2)},\theta_{1}^{(1)})$ does not contain $\lambda_{k}$ but $[\theta_{1}^{(3)},\theta_{1}^{(2)})$ does (hatched).}
\label{fig:initial}
\end{figure}
\par
%---------------------------------------------------------------------------------------%
\subsection{Computation of the $k$-th eigenpair with validation}\label{sec:app:validate}%
%---------------------------------------------------------------------------------------%
Once an initial interval is set for bisection (Section \ref{sec:app:eff}), it is narrowed down until it contains the $k$-th eigenvalue and some nearby eigenvalues.
We then compute the $k$-th eigenpair along with the other eigenpairs of the interval while validating their index.
The eigenpairs of the interval can be computed by a variety of methods (Section \ref{sec:intro}).
We utilize a modified SI Lanczos method (Section \ref{sec:bp:pre:silanczos}), where approximate eigenvectors $\boldsymbol{x}_{i,2}^{(j)}$ are computed based on \eqref{eq:sigep:vec2} and have $B$-orthogonality, as shown in Proposition \ref{prop:ortho}.
\par
In Proposition \ref{prop:bound}, we explain that an eigenvalue error bound for validating the index of the eigenpairs can be evaluated from the residual vector of the SI Lanczos method at negligible cost.
\par
\begin{proposition}\label{prop:bound}
With the same notation as (\ref{eq:sigep:lanczos}--\ref{eq:sigep:value}) in Section \ref{sec:bp:pre:silanczos}, there is an eigenvalue $\lambda_{{l}_{i}^{(j)}}$ of problem \eqref{eq:kep} such that:
\begin{align}\label{eq:sigep:errbnd}%
\lambda_{{l}_{i}^{(j)}} \in \Gamma_{i}^{(j)} \equiv [\lambda_{i}^{(j)}-\eta_{i}^{(j)},\lambda_{i}^{(j)}+\eta_{i}^{(j)}],\quad \eta_{i}^{(j)} \equiv \frac{1}{| \tilde{\theta}_{i}^{(j)} |} \cdot \frac{| \tilde{\beta}_{j} \boldsymbol{e}_{j}^{\mathrm{T}} \tilde{\boldsymbol{y}}_{i}^{(j)} / \tilde{\theta}_{i}^{(j)} |}{\sqrt{1 + | \tilde{\beta}_{j} \boldsymbol{e}_{j}^{\mathrm{T}} \tilde{\boldsymbol{y}}_{i}^{(j)} / \tilde{\theta}_{i}^{(j)}|^{2}}}.
\end{align}%
\end{proposition}
\begin{proof}
According to a classical result of the perturbation theory \cite[Theorem 15.9.1]{parlett1998symmetric}, there is an eigenvalue $\lambda_{l}$ of problem \eqref{eq:kep} for a scalar $\mu$ and non-zero vector $\boldsymbol{u}$ such that:
\begin{align}\label{eq:sigep:errbnd:theo}%
|\lambda_{l} - \mu | \le \frac{\| (A - \mu B)\boldsymbol{u} \|_{B^{-1}}}{\| B \boldsymbol{u} \|_{B^{-1}}}.
\end{align}
By substituting $(\mu,\boldsymbol{u})$ in \eqref{eq:sigep:errbnd:theo} with an approximate eigenpair $(\lambda_{i}^{(j)},\boldsymbol{x}_{i,2}^{(j)})$ obtained by the SI Lanczos method, we have from \eqref{eq:sigep:ortho} and residual \eqref{eq:sigep:res2} that:
\begin{align}\label{eq:sigep:errbnd:theo:sub}
|\lambda_{{l}_{i}^{(j)}} - \lambda_{i}^{(j)} | \le \frac{\| (A - \lambda_{i}^{(j)} B)\boldsymbol{x}_{i,2}^{(j)} \|_{B^{-1}}}{\| B \boldsymbol{x}_{i,2}^{(j)} \|_{B^{-1}}} = \frac{\| \boldsymbol{r}_{i,2}^{(j)} \|_{B^{-1}}}{\| \boldsymbol{x}_{i,2}^{(j)} \|_{B}} = \frac{1}{| \tilde{\theta}_{i}^{(j)} |} \cdot \frac{| \tilde{\beta}_{j} \boldsymbol{e}_{j}^{\mathrm{T}} \tilde{\boldsymbol{y}}_{i}^{(j)} / \tilde{\theta}_{i}^{(j)} |}{\sqrt{1 + | \tilde{\beta}_{j} \boldsymbol{e}_{j}^{\mathrm{T}} \tilde{\boldsymbol{y}}_{i}^{(j)} / \tilde{\theta}_{i}^{(j)}|^{2}}}.
\end{align}
By rewriting inequality \eqref{eq:sigep:errbnd:theo:sub}, we have an eigenvalue error bound $\Gamma_{i}^{(j)}$ for $\lambda_{i}^{(j)}$ that includes an eigenvalue $\lambda_{{l}_{i}^{(j)}}$ of problem \eqref{eq:kep}.
\end{proof}
\par
When bound \eqref{eq:sigep:errbnd} becomes sufficiently narrow, we can associate $\lambda_{i}^{(j)}$ with an eigenvalue of problem \eqref{eq:kep} and thus validate the index of the approximate eigenpairs.
Here, assume that an interval $[\sigma_{\mathrm{lower}},\sigma_{\mathrm{upper}})$ contains $m$ eigenvalues.
If there are $m$ error bounds in the interval that are mutually disjoint, as illustrated in Figure \ref{fig:bound}\subref{fig:bound:true}, then each bound contains only one of the $m$ eigenvalues of the interval.
When the approximate eigenpairs have one-to-one correspondence with the eigenvalues of the interval, the index of each approximate eigenpair can be validated readily because we already know the index range of the eigenpairs of the interval from the bisection.
\par
\begin{figure}[htbp]
\centering
\begin{minipage}{1.0\linewidth}
\centering
\includegraphics[width=0.8\linewidth]{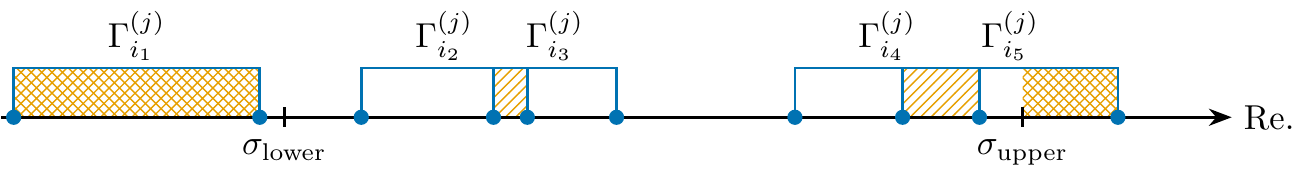}
\subcaption{$\Gamma_{i_{1}}^{(j)}$ and $\Gamma_{i_{5}}^{(j)}$ are not included in the interval (crosshatched).
$\Gamma_{i_{2}}^{(j)}$ and $\Gamma_{i_{3}}^{(j)}$ overlap, so do $\Gamma_{i_{4}}^{(j)}$ and $\Gamma_{i_{5}}^{(j)}$ (hatched).}\label{fig:bound:false}
\end{minipage}
\begin{minipage}{1.0\linewidth}
\centering
\includegraphics[width=0.8\linewidth]{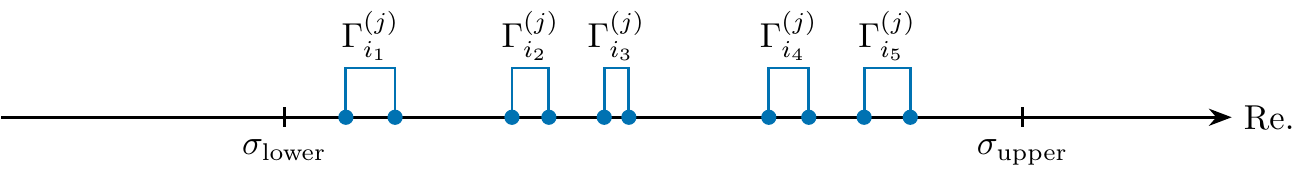}
\subcaption{$m=5$ mutually disjoint bounds in the interval}\label{fig:bound:true}
\end{minipage}
\caption{Error bounds \eqref{eq:sigep:errbnd} for $m = 5$}
\label{fig:bound}
\end{figure}
\par
Algorithm \ref{alg:validate} shows an implementation example, where the midpoint $\sigma=(\sigma_{\mathrm{lower}}+\sigma_{\mathrm{upper}})/2$ of the interval is selected as the shift for the SI Lanczos method.
In line \ref{alg:validate:sort}, we sort the approximate eigenpairs $(\lambda_{i}^{(j)},\boldsymbol{x}_{i,2}^{(j)})$ of the method and re-index them for simplicity:
\begin{align}
| \lambda_{1}^{(j)} - \sigma | \le | \lambda_{2}^{(j)} - \sigma | \le \cdots \le | \lambda_{j}^{(j)} - \sigma |,\ \text{or equivalently}\ | \tilde{\theta}_{1}^{(j)} | \ge | \tilde{\theta}_{2}^{(j)} | \ge \cdots \ge | \tilde{\theta}_{j}^{(j)}|.\label{eq:sigep:errbnd:index:approx}
\end{align}
Then, if the following two conditions hold, each $\lambda_{i}^{(j)}$ best approximates a distinct eigenvalue of the interval:
\begin{align}
\text{inclusion:}\quad &\Gamma_{i}^{(j)} \subset [\sigma_{\mathrm{lower}},\sigma_{\mathrm{upper}}) \ \text{for}\ 1 \le i \le m.\label{eq:sigep:errbnd:cond:include}\\
\text{disjointness:}\quad &\Gamma_{i_{1}}^{(j)} \cap \Gamma_{i_{2}}^{(j)} = \varnothing\ \text{for}\ 1 \le i_{1} < i_{2} \le m.\label{eq:sigep:errbnd:cond:disjoint}
\end{align}
When \eqref{eq:sigep:errbnd:cond:include} and \eqref{eq:sigep:errbnd:cond:disjoint} hold, we test for convergence in line \ref{alg:validate:test}.
The algorithm is assumed to reach convergence when the relative residual $2$-norm of each approximate eigenpair of the interval becomes less than a given tolerance $\tau_{\mathrm{res}}.$
We also utilize the following criterion for the convergence test, which we refer to as the relative difference $2$-norm between the approximate eigenvectors at the $(j-1)$-th and $j$-th iterations:
\begin{align}\label{eq:reldiff}
\| \boldsymbol{x}_{i,2}^{(j)} - \boldsymbol{x}_{i,2}^{(j-1)} \|_{2} / \| \boldsymbol{x}_{i,2}^{(j-1)} \|_{2} < \tau_{\mathrm{diff}}\ \text{for}\ 1 \le i \le m.
\end{align}
Here, $\boldsymbol{x}_{i,2}^{(j-1)}$ and $\boldsymbol{x}_{i,2}^{(j)}$ are normalized to satisfy $\| \boldsymbol{x}_{i,2}^{(j-1)} \|_{2} = \| \boldsymbol{x}_{i,2}^{(j)} \|_{2}.$
As discussed in Section \ref{sec:numerical:overview:eigenpair}, criterion \eqref{eq:reldiff} is required because a small relative residual $2$-norm does not necessarily imply that an approximate eigenvector is close to convergence.
After the convergence test, we sort the approximate eigenpairs of the interval to the original order in line \ref{alg:validate:sortback} to obtain the $k$-th eigenpair in line \ref{alg:validate:index}.
\par
\begin{algorithm2e}[H]%
\SetNlSty{textrm}{}{}\SetNlSkip{0.625em}\DontPrintSemicolon
\SetKwInOut{Input}{Input}\SetKwInOut{Output}{Output}
\SetKwFor{For}{for}{do}{{end for}}
\SetKwIF{If}{ElseIf}{Else}{if}{then}{else if}{else}{{end if}}
\Input{matrices $A,B$ of generalized eigenvalue problem \eqref{eq:kep}, target index $k,$\\
interval $[\sigma_{\mathrm{lower}},\sigma_{\mathrm{upper}})$ containing the $k$-th eigenvalue,\\
$\nu_{\mathrm{lower}}=\nu_{\sigma_{\mathrm{lower}}}(A,B),\ \nu_{\mathrm{upper}}=\nu_{\sigma_{\mathrm{upper}}}(A,B),$\\
tolerance $\tau_{\mathrm{res}}$ for relative residual $2$-norm, tolerance $\tau_{\mathrm{diff}}$ for relative difference $2$-norm.}
\Output{approximate eigenpair $(\hat{\lambda}_{k},\hat{\boldsymbol{x}}_{k}),$ where $\| (A - \hat{\lambda}_{k}B)\hat{\boldsymbol{x}}_{k} \|_{2} / \| \hat{\boldsymbol{x}}_{k} \|_{2} < \tau_{\mathrm{res}}.$}
$l:=k - \nu_{\mathrm{lower}},\ m:=\nu_{\mathrm{upper}} - \nu_{\mathrm{lower}},\ \sigma:=(\sigma_{\mathrm{lower}}+\sigma_{\mathrm{upper}})/2,$\;
set a random starting vector $\tilde{\boldsymbol{v}}_{1},\ \tilde{\boldsymbol{v}}_{1} := \tilde{\boldsymbol{v}}_{1} / \| \tilde{\boldsymbol{v}}_{1} \|_{B^{-1}},$\;
\For{$j=1,2,\ldots$}{
	compute $j$-step SI Lanczos decomposition \eqref{eq:sigep:lanczos} with reorthogonalization,\;
	\If{$j \ge m$}{
	solve standard eigenvalue problem \eqref{eq:sigep:triep},\;
	\lFor{$i=1$ \KwTo j}{compute approximate eigenpairs $(\lambda_{i}^{(j)},\boldsymbol{x}_{i,2}^{(j)})$ as \eqref{eq:sigep:value} and \eqref{eq:sigep:vec2},}
	sort $(\lambda_{i}^{(j)},\boldsymbol{x}_{i,2}^{(j)})$ with $1 \le i \le j$ to index as \eqref{eq:sigep:errbnd:index:approx},\label{alg:validate:sort}\;
	\lFor{$i=1$ \KwTo m}{compute the $i$-th eigenvalue error bound \eqref{eq:sigep:errbnd},}
		\lIf{\upshape \eqref{eq:sigep:errbnd:cond:include} and \eqref{eq:sigep:errbnd:cond:disjoint} and ($\| (A - \lambda_{i}^{(j)}B)\boldsymbol{x}_{i,2}^{(j)} \|_{2} / \| \boldsymbol{x}_{i,2}^{(j)} \|_{2} < \tau_{\mathrm{res}}$ for $1 \le i \le m$) and \eqref{eq:reldiff}}{break,}\label{alg:validate:test}
		}
}
sort $(\lambda_{i}^{(j)},\boldsymbol{x}_{i,2}^{(j)})$ with $1 \le i \le m$ to index in increasing order $\lambda_{1}^{(j)} < \lambda_{2}^{(j)} < \cdots <\lambda_{m}^{(j)},$\label{alg:validate:sortback}\;
$(\hat{\lambda}_{k},\hat{\boldsymbol{x}}_{k}):=(\lambda_{l}^{(j)},\boldsymbol{x}_{l,2}^{(j)}).$\label{alg:validate:index}\;
\caption{Computation of the $k$-th eigenpair\label{alg:validate}}
\end{algorithm2e}
\par
%--------------------------------------------------------------------%
\subsection{Application to large-scale problems}\label{sec:app:large}%
%--------------------------------------------------------------------%
We utilize a sparse direct linear solver for computing $\nu_{\sigma}(A,B)$ of large sparse $A,B.$
Algorithm \ref{alg:bisect:modified} shows a modification of Algorithm \ref{alg:bisect} to narrow down an initial interval.
The main difference can be found in line \ref{alg:bisect:modified:sparse}, where a fill-reducing ordering $Q$ is utilized to handle large sparse matrices, in addition to permutation $P$ for numerical stability.
In contrast to $P,$ ordering $Q$ is independent from shift $\sigma$ because it depends on only the sparsity structure of $A -\sigma B.$
Thus, once ordering and symbolic factorization are obtained for some shifted matrix, they can be recycled for other shifted matrices with varying shifts in the subsequent iterations.
Computation of $\nu_{\sigma}(A,B)$ in Algorithm \ref{alg:initial} can be performed in the same manner by utilizing a sparse direct linear solver.
If shifted matrices are known to have a data-sparse (hierarchical low-rank) structure, their factorization can be done fast in a recursive manner \cite{lintner2004eigenvalue,benner2012computing,xi2014fast}.
\par
\begin{algorithm2e}[H]%
\SetAlgoRefName{1${}^{\prime}$}
\SetNlSty{textrm}{}{}\SetNlSkip{0.625em}\DontPrintSemicolon
\SetKwInOut{Input}{Input}\SetKwInOut{Output}{Output}
\SetKwIF{If}{ElseIf}{Else}{if}{then}{else if}{else}{{end if}}
\SetKwIF{RIf}{RElseIf}{RElse}{repeat until}{}{}{}{{}}
\SetKwComment{tcp}{$\triangleright$~}{}\SetCommentSty{textrm}
\Input{matrices $A,B$ of generalized eigenvalue problem \eqref{eq:kep}, target index $k,$\\
initial interval $[\sigma_{\mathrm{lower}},\sigma_{\mathrm{upper}})$ containing the $k$-th eigenvalue,\\
$\nu_{\mathrm{lower}}=\nu_{\sigma_{\mathrm{lower}}}(A,B),\ \nu_{\mathrm{upper}}=\nu_{\sigma_{\mathrm{upper}}}(A,B),$ stopping criterion $m_{\mathrm{max}}.$}
\Output{interval $[\sigma_{\mathrm{lower}},\sigma_{\mathrm{upper}})$, $\nu_{\mathrm{lower}},\ \nu_{\mathrm{upper}}.$}
\RIf{$\nu_{\mathrm{upper}} - \nu_{\mathrm{lower}} \le m_{\mathrm{max}}$\label{alg:bisect:modified:criterion}\tcp*[f]{\makebox[23em]{$\nu_{\mathrm{upper}} - \nu_{\mathrm{lower}}$: number of eigenvalues in the interval\hfill}}}{
	$\sigma:=(\sigma_{\mathrm{lower}}+\sigma_{\mathrm{upper}})/2,$\label{alg:bisect:modified:root}\;
	$LDL^{\mathrm{H}}\leftarrow PQ(A-\sigma B)Q^{\mathrm{T}}P^{\mathrm{T}},$\tcp*{\makebox[23em]{$Q:$ fill-reducing ordering for sparse matrices\hfill}}\label{alg:bisect:modified:sparse}
	$\nu:=\nu_{0}(D,I),$\;
	\leIf{$k \le \nu$}{$\sigma_{\mathrm{upper}}:=\sigma,\ \nu_{\mathrm{upper}}:=\nu$}{$\sigma_{\mathrm{lower}}:=\sigma,\ \nu_{\mathrm{lower}}:=\nu.$}}	
\caption{Spectral bisection for the $k$-th eigenvalue problem\label{alg:bisect:modified}}
\end{algorithm2e}
\par
%---------------------------------------------------------------------%
\subsection{Overview of the three-stage algorithm}\label{sec:app:over}%
%---------------------------------------------------------------------%
Algorithm \ref{alg:frame} presents a three-stage algorithm for solving the $k$-th eigenvalue problem.
We first run Algorithm \ref{alg:initial} to set an interval $[\sigma_{\mathrm{lower}},\sigma_{\mathrm{upper}})$ containing the $k$-th eigenvalue.
We then run Algorithm \ref{alg:bisect:modified} to narrow down the interval until it contains less than or equal to $m_{\mathrm{max}}$ eigenvalues that include $\lambda_{k}.$
Compared with further narrowing down the interval to isolate $\lambda_{k}$ from the other eigenvalues, approximately $\log_{2}m_{\mathrm{max}}$ bisection iterations can be reduced.
Finally, we run Algorithm \ref{alg:validate} to compute the $k$-th eigenpair along with the other eigenpairs of the interval while validating their index.
\par
\begin{algorithm2e}[H]%
\SetNlSty{textrm}{}{}\SetNlSkip{0.625em}\DontPrintSemicolon
\SetKwInOut{Input}{Input}\SetKwInOut{Output}{Output}
\SetKwFor{For}{for}{do}{{end for}}
\SetKwIF{If}{ElseIf}{Else}{if}{then}{else if}{else}{{end if}}
\SetKwIF{RIf}{RElseIf}{RElse}{repeat until}{}{}{}{{}}
\SetKwComment{tcp}{$\triangleright$~}{}\SetCommentSty{textrm}
\Input{matrices $A,B$ of generalized eigenvalue problem \eqref{eq:kep}, target index $k,$\\
stopping criterion $m_{\mathrm{max}}$ for spectral bisection,\\
tolerance $\tau_{\mathrm{res}}$ for relative residual $2$-norm, tolerance $\tau_{\mathrm{diff}}$ for relative difference $2$-norm.}
\Output{approximate eigenpair $(\hat{\lambda}_{k},\hat{\boldsymbol{x}}_{k}).$}
run Algorithm \ref{alg:initial} to set an interval $[\sigma_{\mathrm{lower}},\sigma_{\mathrm{upper}})$ containing $\lambda_{k}$ and obtain $\nu_{\mathrm{lower}} = \nu_{\sigma_{\mathrm{lower}}}(A,B)$ and $\nu_{\mathrm{upper}} = \nu_{\sigma_{\mathrm{upper}}}(A,B),$\label{alg:frame:initial}\;
run Algorithm \ref{alg:bisect:modified} to narrow down the interval $[\sigma_{\mathrm{lower}},\sigma_{\mathrm{upper}})$ until it contains less than or equal to $m_{\mathrm{max}}$ eigenvalues that include $\lambda_{k},$\;
run Algorithm \ref{alg:validate} to obtain the $k$-th eigenpair $(\hat{\lambda}_{k},\hat{\boldsymbol{x}}_{k})$ with its relative residual and difference $2$-norms less than $\tau_{\mathrm{res}}$ and $\tau_{\mathrm{diff}},$ respectively.\;
\caption{Three-stage algorithm for solving the $k$-th eigenvalue problem \label{alg:frame}}
\end{algorithm2e}
\par
Here, $m_{\mathrm{max}}$ is an important parameter that influences the overall performance of the three-stage algorithm because there is a trade-off between the second and third stages (Algorithms \ref{alg:bisect:modified} and \ref{alg:validate}), i.e., greater $m_{\mathrm{max}}$ results in fewer required bisection iterations, while more iterations are required for the SI Lanczos method.
$m_{\mathrm{max}}$ was set the same for each problem in the numerical experiments.
Tuning this parameter will be the focus of future work.
\par
In the presence of multiple eigenvalues or a cluster of eigenvalues around $\lambda_{k},$ modification to Algorithm \ref{alg:frame} is necessary  because the algorithm is ineffective in detecting them and may end up in misconvergence; the stopping criterion (line \ref{alg:bisect:modified:criterion} of Algorithm \ref{alg:bisect:modified}) for bisection and a convergence criterion \eqref{eq:sigep:errbnd:cond:disjoint} for the SI Lanczos method may not be satisfied. In addition, from the SI Lanczos method, only a one dimensional representation is obtained for the eigenspace corresponding to a multiple eigenvalue.
To detect multiple or a cluster of eigenvalues during bisection, the length of the interval can be used along with the current stopping criterion.
If detected, they and their corresponding eigenspace can be computed by, e.g., a block SI Lanczos method \cite{grimes1994shifted} whose block size can be determined from the number of eigenvalues in the interval.
When a block eigensolver is used, the convergence criteria (line \ref{alg:validate:test} of Algorithm \ref{alg:validate}) need to be modified accordingly to take account of the multiplicity or the cluster.
Further investigation of the modification will be future work.
\par
%---------------------------------------------------%
\section{Numerical experiments}\label{sec:numerical}%
%---------------------------------------------------%
This section reports the numerical results of several real research problems from electronic structure calculations and a comparison of the proposed three-stage algorithm (Algorithm \ref{alg:frame}) and dense eigensolvers.
In Section \ref{sec:numerical:matrix}, we describe the matrix data used in the numerical experiments.
Section \ref{sec:numerical:comp} provides implementation details of dense eigensolvers and the three-stage algorithm.
The numerical results are reported in Section \ref{sec:numerical:result}.
\par
%---------------------------------------------------%
\subsection{Matrix data}\label{sec:numerical:matrix}%
%---------------------------------------------------%
Table \ref{tab:matrix} shows the matrix data used in the numerical experiments, which were generated by the ELSES quantum mechanical nanomaterial simulator \cite{hoshi2012order} and obtained from the ELSES Matrix Library (\url{http://www.elses.jp/matrix/}).
Here, the two $n \times n$ matrices $A$ and $B$ of each matrix data are real symmetric and real symmetric positive definite, respectively.
$A$ and $B$ have the same sparsity structure shown in Figure \ref{fig:sparse} with \#nz non-zero elements (in their lower triangular part) in Table \ref{tab:matrix}.
The target index $k$ is associated with the HO state.
The entire spectrum is included in the interval $[\lambda_{1},\lambda_{n}].$
The last column of Table \ref{tab:matrix} describes the origin of the matrix data.
\par
\begin{table}[htbp]
\centering
\caption{Matrix data}
\label{tab:matrix}
\begin{tabular}[c]{l|r|r|r|r@{,}l|l}
\hline
\multicolumn{1}{c|}{Data}	&	\multicolumn{1}{c|}{$n$}	&	\multicolumn{1}{c|}{\#nz}			&	\multicolumn{1}{c|}{$k$} & \multicolumn{2}{c|}{$[\lambda_{1},\lambda_{n}]$}		&	\multicolumn{1}{c}{Material}	\\
\hline
APF4686	&	4686	&	53950 &	2343	& $[-1.157$ & $\phantom{1}5.581]$ & \begin{tabular}[c]{@{}l@{}}amorphous-like conjugated polymer,\\ poly(9,9-dioctyl-fluorene) \cite{hoshi2012order}\end{tabular}\\
\hline
AUNW9180	&	9180	&	1783313 	&	5610	& $[-0.210$ & $\phantom{1}0.883]$ &	\begin{tabular}[c]{@{}l@{}}helical multishell gold nanowire\\ with defects \cite{HoshiFujiwara2009}\end{tabular}\\
\hline
CPPE32346	&	32346	&	861764 &	16173	& $[-1.169$ & $\phantom{1}7.953]$ &	\begin{tabular}[c]{@{}l@{}}condensed polymer systems,\\ poly(phenylene-ethynylene) \cite{imachi2016one,hoshi2016extremely}\end{tabular}\\
\hline
NCCS430080	&	430080	&	10696416	&	215040	& $[-1.195$ & $13.602]$ &	\begin{tabular}[c]{@{}l@{}}sp\textsuperscript{2}--sp\textsuperscript{3} nano-composite carbon\\ solid \cite{hoshi2013ten}\\
\end{tabular}\\
\hline
VCNT1512000	&	1512000	&	294953351	&	336000	& $[-1.098$ & $\phantom{1}0.475]$ &	\begin{tabular}[c]{@{}l@{}}vibrating carbon nanotube within\\ a supercell with spd orbitals \cite{PhysRevB.61.7965}\end{tabular}	\\
\hline
\end{tabular}
\end{table}
\begin{figure}[htbp]
\centering
\begin{tabular}{ccccc}
\includegraphics[width=0.173\linewidth]{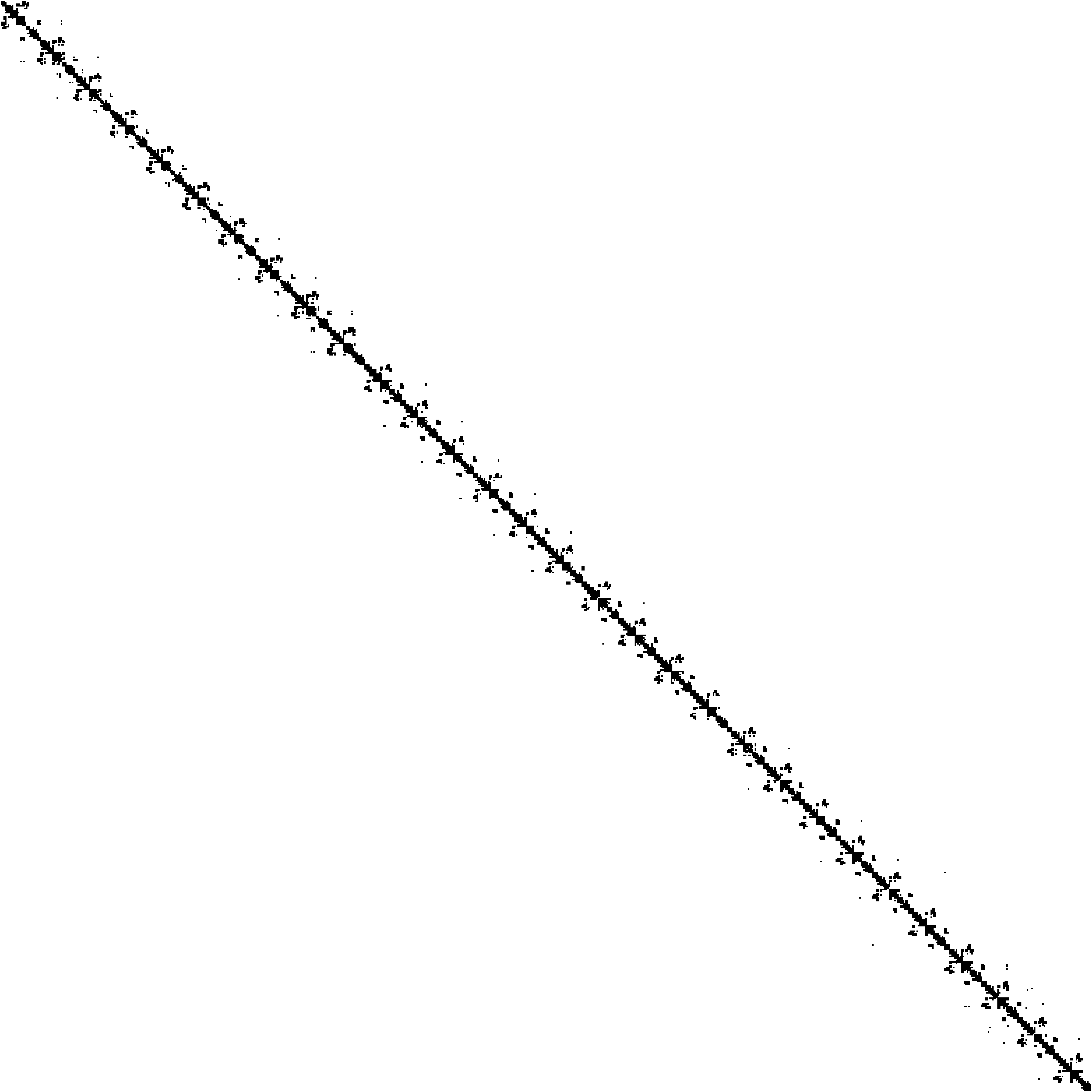}&%
\includegraphics[width=0.173\linewidth]{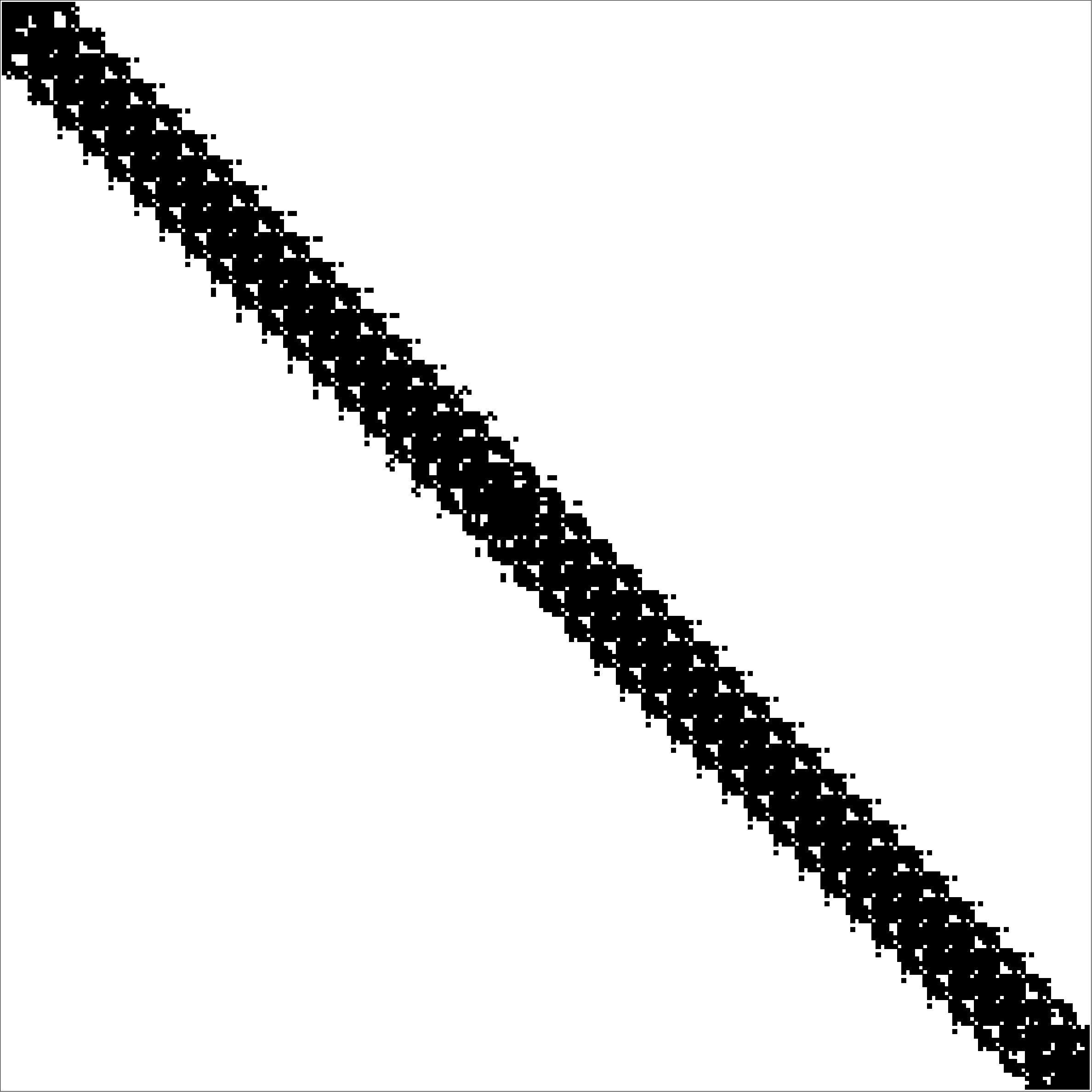}&%
\includegraphics[width=0.173\linewidth]{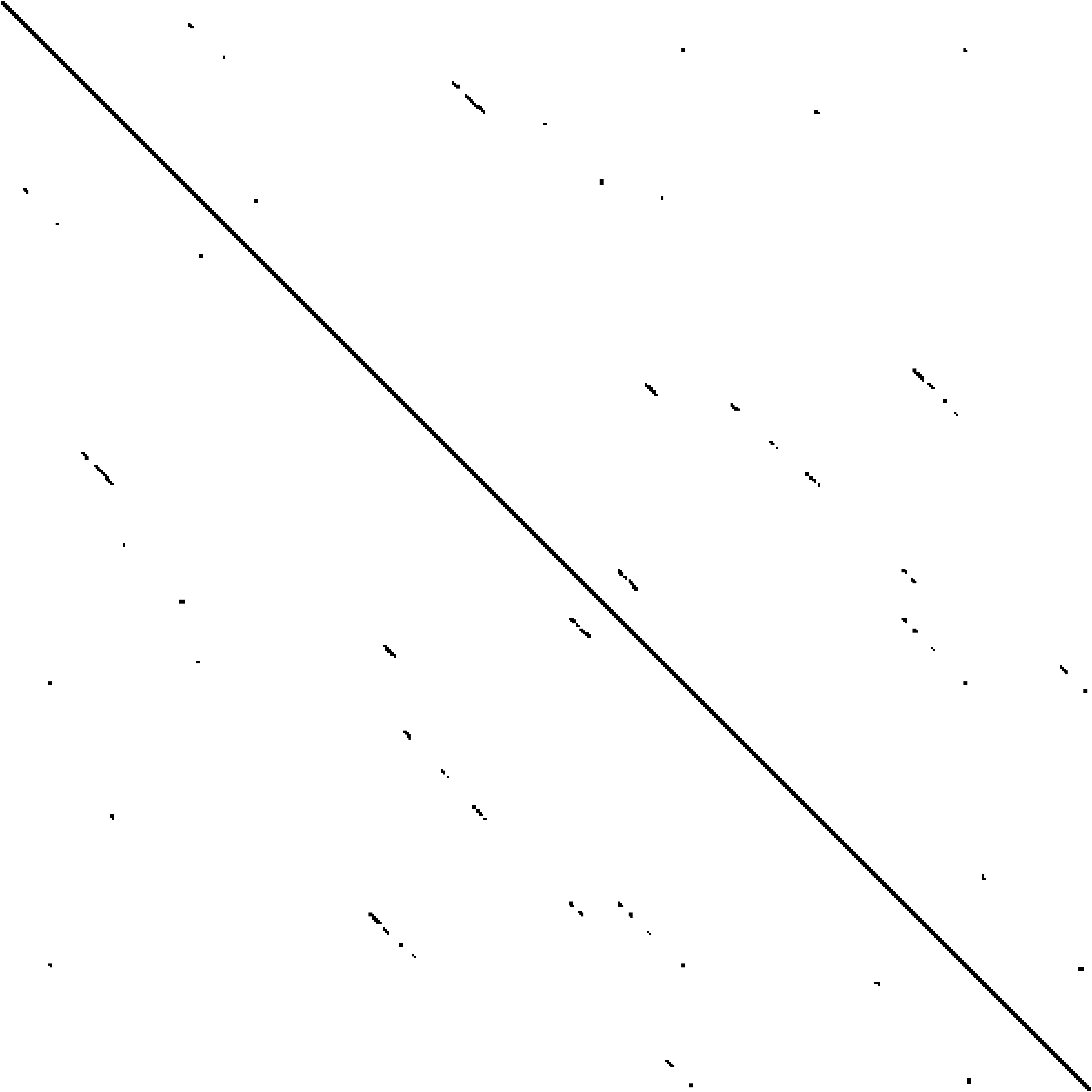}&%
\includegraphics[width=0.173\linewidth]{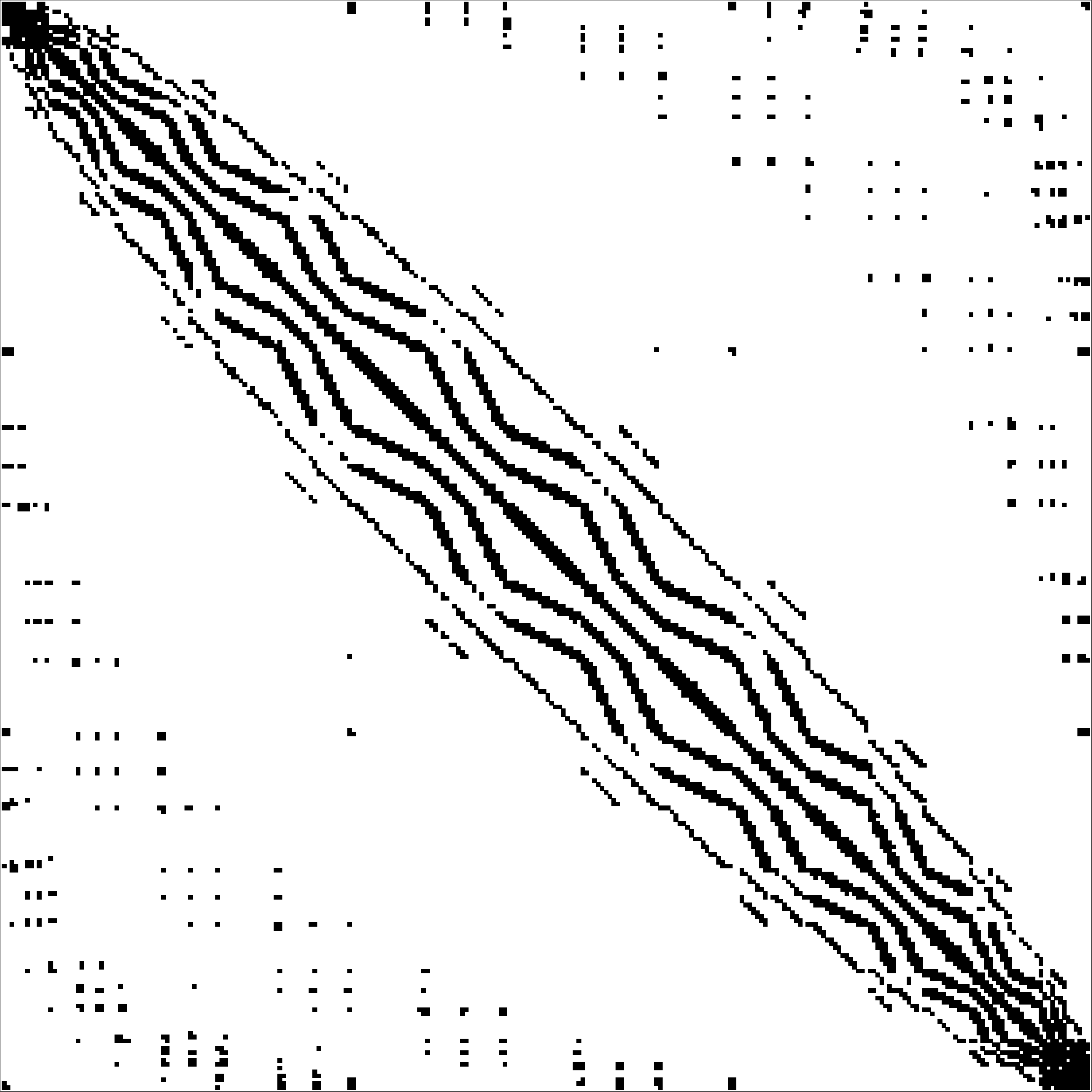}&%
\includegraphics[width=0.173\linewidth]{./008_nz_VCNT1512000.pdf}\\
APF4686&AUNW9180&CPPE32346&NCCS430080&VCNT1512000
\end{tabular}
\caption{Sparsity structures of matrix data}
\label{fig:sparse}
\end{figure}
\par
%------------------------------------------------------------%
\subsection{Implementation details}\label{sec:numerical:comp}%
%------------------------------------------------------------%
As dense eigensolvers, we used the LAPACK \cite{lug1999} and ScaLAPACK \cite{slug1997} routines.
Specifically, the LAPACK \texttt{dsygvd} routine was used to solve APF4686 and AUNW9180.
In \texttt{dsygvd}, a generalized eigenvalue problem is transformed to a standard eigenvalue problem of a tridiagonal matrix, and then eigenpairs of the tridiagonal matrix are computed by the divide and conquer method \cite{cuppen1981divide,gu1995divide,tisseur1999parallel}.
To date, there is no single ScaLAPACK routine to perform the same task as \texttt{dsygvd} in parallel.
Therefore, to solve CPPE32346 and NCCS430080, the ScaLAPACK \texttt{pdpotrf}, \texttt{pdsygst}, \texttt{pdsytrd}, \texttt{pdstedc}, and \texttt{pdormtr} routines were utilized through EigenKernel (\url{https://github.com/eigenkernel/}) \cite{imachi2016hybrid}.
Note that the results for VCNT1512000 are not provided in the present paper because the problem size prevents it from being solved by a dense eigensolver in practical time.
\par
In the three-stage algorithm (Algorithm \ref{alg:frame}), $\nu_{\sigma}(A,B)$ and solution of linear systems in the Lanczos and SI Lanczos methods were computed based on $LDL^{\mathrm{H}}$ factorization by the MUMPS sparse direct linear solver \cite{amestoy2001fully,amestoy2006hybrid} with the METIS fill-reducing ordering \cite{karypis1998fast}.
The bisection narrowed down the initial interval until the number of eigenvalues in the interval became less than or equal to $m_{\mathrm{max}} = 20.$
Tolerances for the relative residual and difference 2-norms in the SI Lanczos method were set to $\tau_{\mathrm{res}} = 10^{-10}$ and $\tau_{\mathrm{diff}} = 10^{-10},$ respectively.
All codes were written in Fortran 90, and the numerical experiments were performed in double-precision.
\par
%-----------------------------------------------%
\subsection{Results}\label{sec:numerical:result}%
%-----------------------------------------------%
In this subsection, we first compare the $k$-th eigenpair computed by Algorithm \ref{alg:frame} with that obtained by the dense eigensolvers described in Section \ref{sec:numerical:comp} and report the computation time with some details about the computational environment and implementation.
Then, in Sections \ref{sec:numerical:overview:initial} to \ref{sec:numerical:overview:eigenpair}, we present the detailed results of each algorithm (i.e., Algorithms \ref{alg:bisect:modified}, \ref{alg:initial}, and \ref{alg:validate}) of the three-stage algorithm.
\par
Table \ref{tab:kthev} compares the $k$-th eigenpair of the three-stage algorithm $(\hat{\lambda}_{k},\hat{\boldsymbol{x}}_{k})$ and that of the dense eigensolvers $(\lambda_{k}^{\mathrm{(d)}},\boldsymbol{x}_{k}^{\mathrm{(d)}}).$
As can be seen, at least 15 digits were the same for $\hat{\lambda}_{k}$ and $\lambda_{k}^{\mathrm{(d)}}.$
The last column shows the relative error $2$-norm, where $\hat{\boldsymbol{x}}_{k}$ and $\boldsymbol{x}_{k}^{\mathrm{(d)}}$ were normalized to satisfy $\| \hat{\boldsymbol{x}}_{k} \|_{2} = \| \boldsymbol{x}_{k}^{\mathrm{(d)}} \|_{2}.$
The error norm had an order of magnitude less than $-10,$ indicating that the $k$-th eigenvector of the three-stage algorithm agrees well with that of the dense eigensolvers.
\par
\begin{table}[htbp]
\centering
\caption{$k$-th eigenpair}\label{tab:kthev}
\begin{tabular}[c]{l|r|r|r|c|c}
\hline
\multicolumn{1}{c|}{Data}	&	\multicolumn{1}{c|}{$k$}	&	\multicolumn{1}{c|}{$\hat{\lambda}_{k}$}	&	\multicolumn{1}{c|}{$\lambda_{k}^{\mathrm{(d)}}$}	&	$\frac{| \hat{\lambda}_{k} - \lambda_{k}^{\mathrm{(d)}} |}{| \lambda_{k}^{\mathrm{(d)}} |}$ & $\frac{\| \hat{\boldsymbol{x}}_{k} - \boldsymbol{x}_{k}^{\mathrm{(d)}} \|_{2}}{\| \boldsymbol{x}_{k}^{\mathrm{(d)}} \|_{2}}$\\
\hline
APF4686	&	2343	&	$-0.4258775547956963$	&	$-0.4258775547956963$	&	$0$	& $4\times10^{-14}$	\\
AUNW9180	&	5610	&	$0.1305388835941175$	&	$0.1305388835941177$	&	$2\times10^{-15}$	 & $1\times10^{-12}$	\\
CPPE32346	&	16173	&	$-0.4332412034185730$	&	$-0.4332412034185731$	&	$2\times10^{-16}$	 & $7\times10^{-14}$	\\
NCCS430080	&	215040	&	$-0.3689638375042860$	&	$-0.3689638375042869$	&	$2\times10^{-15}$	 & $4\times10^{-11}$	\\
VCNT1512000	&	336000	&	$-0.5517499297808635$	&	\multicolumn{1}{c|}{n/a}	&	n/a	& n/a\\
\hline
\end{tabular}
\end{table}
\par
Table \ref{tab:comp} shows the total computation time and computational resources consumed by Algorithm \ref{alg:frame} (Alg.~4), its variant (Ger.), and the dense eigensolvers (Dense).
In the variant, line \ref{alg:frame:initial} of Algorithm \ref{alg:frame} was changed to set an interval including the entire spectrum based on the Gershgorin circle theorem{\interfootnotelinepenalty=10000\footnote{%
Instead of some Gershgorin-type theorem, an inclusion set of the spectrum of $B^{-1} A$ was computed based on the original theorem because the diagonal dominance of $A$ and $B$ (described in Section \ref{sec:app:eff}) does not hold for all matrix data.
To compute the inclusion set, columns of $B^{-1} A$ were obtained by solving linear systems with MUMPS, and then Gershgorin disks were calculated from the columns.
In the VCNT1512000 case, the linear systems were solved in single-precision to reduce the time to solution.%
}.}
Here, \#Core is the number of cores used in the experiments, and superscripts \textsuperscript{(w)} and \textsuperscript{(K)} represent a workstation and the K computer, respectively.
Memory indicates peak memory usage.
Actual measurement of the memory usage was performed using the GNU \texttt{time} command.
Estimation (in italics) shows the memory required to store $4n^{2}$ double-precision numbers, which is based on the memory requirement of the LAPACK \texttt{dsygvd} routine.
\par
\begin{table}[htbp]
\centering
\caption{Computation time and computational resources consumed by Algorithm \ref{alg:frame}, its variant, and dense eigensolvers}
\label{tab:comp}
\begin{threeparttable}[b]
\begin{tabular}[c]{l|r|r|r|r|r|r|r|r|r}
\hline
\multicolumn{1}{c|}{\multirow{2}{*}{Data}} & \multicolumn{3}{c|}{Time~(s)} & \multicolumn{3}{c|}{\#Core} & \multicolumn{3}{c}{Memory~(MB)}\\
\cline{2-10}
& \multicolumn{1}{c|}{Alg.~4} & \multicolumn{1}{c|}{Ger.} & \multicolumn{1}{c|}{Dense} & \multicolumn{1}{c|}{Alg.~4} & \multicolumn{1}{c|}{Ger.} & \multicolumn{1}{c|}{Dense}	& \multicolumn{1}{c|}{Alg.~4} & \multicolumn{1}{c|}{Ger.} & \multicolumn{1}{c}{Dense}\\
\hline
APF4686	& 0.3 & 0.8 & 82.1 & \multicolumn{1}{c|}{\multirow{5}{*}{1\tnote{(w)}$\phantom{00}$}} & \multicolumn{1}{c|}{\multirow{5}{*}{1\tnote{(w)}$\phantom{00}$}} & 1\tnote{(w)}$\phantom{00}$ & 13 & 12 & 629\\
AUNW9180& 19.4 & 69.0 & 655.2 & & & 1\tnote{(w)}$\phantom{00}$ & 267 & 245 &2836\\
CPPE32346& 4.4 & 194.0 & 1366.8 & & & 32\tnote{(K)}$\phantom{00}$ & 121 & 116
 & \textit{33480}\\
NCCS430080	& 2024 & 109597 & 10586& & & 180000\tnote{(K)}$\phantom{00}$ & 5069 & 5055 & \textit{5919002}\\
VCNT1512000& 5132 & 602525 & n/a& & & n/a & 38242 & 31618 & n/a\\
\hline
\end{tabular}
\begin{tablenotes}
\item[(w)] workstation with Xeon E5-2690 (2.90 GHz)
\item[(K)] K computer with SPARC64 VIIIfx (2.00 GHz) and Tofu interconnect
\end{tablenotes}
\end{threeparttable}
\end{table}
\par
Details about the computation time and implementation of the three-stage algorithm are shown in Figure \ref{fig:time}, where the total time is scaled to one.
As described in the figure legend, the three-stage algorithm consists of Algorithms \ref{alg:bisect:modified}--\ref{alg:validate}, and our implementation can be divided into the following seven major computational tasks.
(\rom{1}) $B$ is preprocessed in a symbolic manner to produce a fill-reducing ordering and an elimination tree for its $LDL^{\mathrm{H}}$ factorization.
The ordering is recycled for the $LDL^{\mathrm{H}}$ factorization of shifted matrices $A -\sigma B$ because matrices $A$ and $B$ have the same sparsity structure in our numerical experiments.
(\rom{2}) Based on the symbolic factorization, the numerical factorization of $B$ is computed to solve linear systems in the Lanczos method.
(\rom{3}) Ritz values are computed to set an initial interval.
(\rom{4}--\rom{6}) Numerical factorization of shifted matrices is computed to set an initial interval, bisect the interval, and solve the linear systems in the SI Lanczos method.
(\rom{7}) The $k$-th eigenpair is computed.
\par
As can be seen in Figure \ref{fig:time}, Algorithm \ref{alg:bisect:modified} dominates computation time as the problem size increases.
This is because, as the problem size increases, more bisection iterations are expected to be required to narrow down an initial interval in order to make the number of eigenvalues in the interval less than or equal to $m_{\mathrm{max}}=20,$ which is the same value regardless of the problem size.
\par
\begin{figure}[htbp]
\centering
\includegraphics[width=1.0\linewidth]{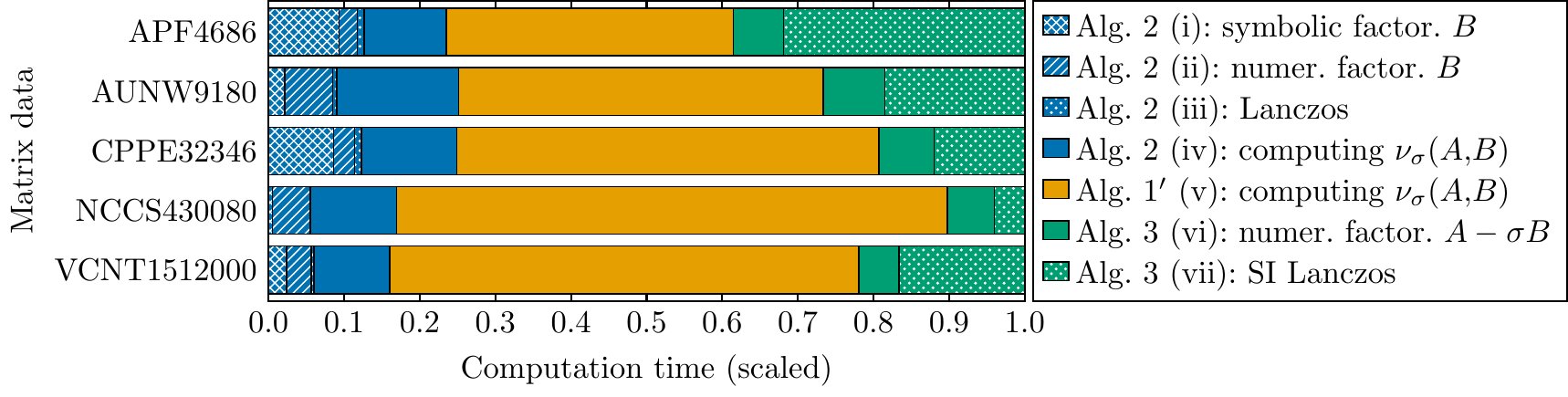}%
\caption{Computation time of algorithms and computational tasks}
\label{fig:time}
\end{figure}
\par
Figure \ref{fig:scaling} shows the relationship between total computation time and the number of non-zero elements in factors $L$ and $D$ of $LDL^{\mathrm{H}}$ factorization, denoted \#nzf, in log--log scale.
Here, \#nzf is the average of the factorizations of $B$ and shifted matrices $A-\sigma B$ with varying $\sigma.$
The dotted line in the figure is of slope one, which corresponds to the linear scaling $O(\mathrm{\#nzf}).$
As can be seen, computation time is proportional to \#nzf (the slope for linear least squares fitting of the data points is 1.06).
This is because the most time-consuming tasks in the three-stage algorithm, i.e., computation of $\nu_{\sigma}(A,B)$ and solving linear systems in the Lanczos and SI Lanczos methods, are performed based on factorizations by a sparse direct linear solver.
Generally, the estimation of \#nzf can be obtained in the symbolic factorization stage, which can be utilized to predict total computation time.
\par
\begin{figure}[htbp]
\centering
\includegraphics[width=0.40\linewidth]{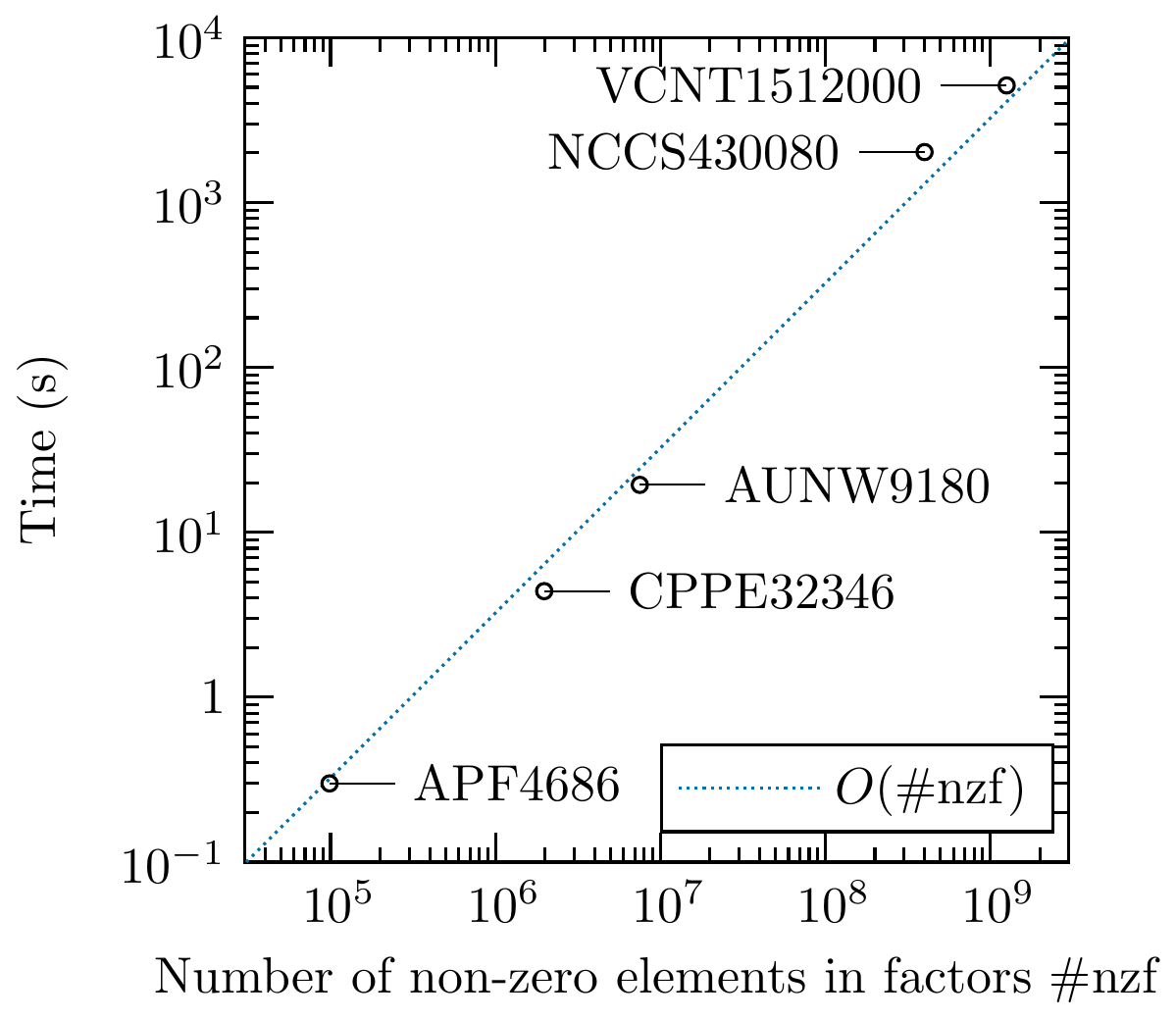}%
\caption{Computation time vs. number of non-zero elements in factors $L$ and $D$ of $LDL^{\mathrm{H}}$ factorization}
\label{fig:scaling}
\end{figure}
\par
%---------------------------------------------------------------------%
\subsubsection{Initial interval}\label{sec:numerical:overview:initial}%
%---------------------------------------------------------------------%
Table \ref{tab:initial} shows the initial interval $[\sigma_{\mathrm{lower}},\sigma_{\mathrm{upper}})$ obtained by Algorithm \ref{alg:initial}.
The fourth column shows the length $\sigma_{\mathrm{upper}} - \sigma_{\mathrm{lower}}$ of the interval, which contains $m$ eigenvalues with their index ranging from $i_{\mathrm{lower}}$ to $i_{\mathrm{upper}}.$
The seventh column shows the average gap between the eigenvalues in the interval defined as $\mathrm{Length}/m.$
In the last column, we compare the length of the initial interval with that of $[\lambda_{1},\lambda_{n}]$ in Table \ref{tab:matrix} by calculating the ratio of $(\lambda_{n}-\lambda_{1})/\mathrm{Length}.$
The initial intervals were 3.2 to 29.5 times narrower than $[\lambda_{1},\lambda_{n}],$ i.e., the tightest interval that can be obtained from some Gershgorin-type theorem in general.
Note that all problems required two iterations of Algorithm \ref{alg:initial} (thus, two $LDL^{\mathrm{H}}$ factorizations) to set the interval, which is the minimum required iterations to obtain an interval validated as containing $\lambda_{k}.$
\par
\begin{table}[htbp]
\centering
\caption{Initial interval}
\label{tab:initial}
\begin{tabular}[c]{l|r|r@{,}l|c|r@{,}l|r|c|c}
\hline
\multicolumn{1}{c|}{Data}	&	\multicolumn{1}{c|}{$k$}	&	\multicolumn{2}{c|}{$[\sigma_{\mathrm{lower}},\sigma_{\mathrm{upper}})$}			&	Length	&	\multicolumn{2}{c|}{$[i_{\mathrm{lower}},i_{\mathrm{upper}}]$}			&	\multicolumn{1}{c|}{$m$}	&	Gap	& Ratio \\
\hline
APF4686	&	2343	&	$[-0.777$	&	$\phantom{-{}}0.475)$	&	1.252 	&	[751	&	3458]	&	2708	&	$5\times10^{-4}$	& $\phantom{2}$5.4 \\
AUNW9180	&	5610	&	$[-0.079$	&	$\phantom{-{}}0.186)$	&	0.265 	&	[877	&	5853]	&	4977	&	$5\times10^{-5}$	& $\phantom{2}$4.1 \\
CPPE32346	&	16173	&	$[-0.731$	&	$\phantom{-{}}1.023)$	&	1.754 	&	[5586	&	26409]	&	20824	&	$8\times10^{-5}$	& $\phantom{2}$5.2 \\
NCCS430080	&	215040	&	$[-0.777$	&	$-0.275)$	&	0.502 	&	[64252	&	224635]	&	160384	&	$3\times10^{-6}$	& 29.5 \\
VCNT1512000	&	336000	&	$[-0.920$	&	$-0.429)$	&	0.491 	&	[84320	&	422420]	&	338101	&	$1\times10^{-6}$	& $\phantom{2}$3.2 \\
\hline
\end{tabular}
\end{table}
\par
%-------------------------------------------------------------%
\subsubsection{Bisection}\label{sec:numerical:overview:bisect}%
%-------------------------------------------------------------%
The initial interval in Table \ref{tab:initial} was narrowed down to the interval $[\sigma_{\mathrm{lower}},\sigma_{\mathrm{upper}})$ in Table \ref{tab:intervalbi} based on Algorithm \ref{alg:bisect:modified}.
Figure \ref{fig:bisection} shows the number of eigenvalues in the interval after each bisection iteration in log scale.
The horizontal dotted line in the figure indicates the stopping criterion $m_{\mathrm{max}}=20$ for the bisection.
In most cases, the number of eigenvalues was approximately halved after each iteration.
However, the number remained unchanged after the fifth iteration of APF4686 and the sixth iteration of CPPE32346.
In addition, there was a sharp decrease in the number of eigenvalues at the final iteration of CPPE32346, in which the number decreased by more than an order of magnitude.
This convergence behavior implies that eigenvalues are distributed in a highly non-uniform manner and that there are clusters of eigenvalues or large gaps between eigenvalues.
Indeed, in the CPPE32346 case, Gap in Table \ref{tab:intervalbi} is approximately 40 times greater than that shown in Table \ref{tab:initial}.
\par
\begin{table}[htbp]
\centering
\caption{Interval narrowed down by bisection}
\label{tab:intervalbi}
\begin{tabular}[c]{l|r|r@{,}l|c|r@{,}l|r|c}
\hline
\multicolumn{1}{c|}{Data}	&	\multicolumn{1}{c|}{$k$}	&	\multicolumn{2}{c|}{$[\sigma_{\mathrm{lower}},\sigma_{\mathrm{upper}})$}			&	Length	&	\multicolumn{2}{c|}{$[i_{\mathrm{lower}},i_{\mathrm{upper}}]$}			&	\multicolumn{1}{c|}{$m$}	&	Gap	\\
\hline
APF4686	&	2343	&	$[-0.44450$	&	$-0.42494)$	&	0.01956 	&	[2334	&	2343]	&	10	&	$2\times10^{-3}$	\\
AUNW9180	&	5610	&	$[\phantom{-{}}0.12826$	&	$\phantom{-{}}0.13240)$	&	0.00414 	&	[5601	&	5615]	&	15	&	$3\times10^{-4}$	\\
CPPE32346	&	16173	&	$[-0.43657$	&	$-0.42971)$	&	0.00685 	&	[16172	&	16173]	&	2	&	$3\times10^{-3}$	\\
NCCS430080	&	215040	&	$[-0.36897$	&	$-0.36884)$	&	0.00013 	&	[215040	&	215049]	&	10	&	$1\times10^{-5}$	\\
VCNT1512000	&	336000	&	$[-0.55178$	&	$-0.55166)$	&	0.00012 	&	[335995	&	336010]	&	16	&	$8\times10^{-6}$	\\
\hline
\end{tabular}
\end{table}
\begin{figure}[htbp]
\centering
\includegraphics[width=1.0\linewidth]{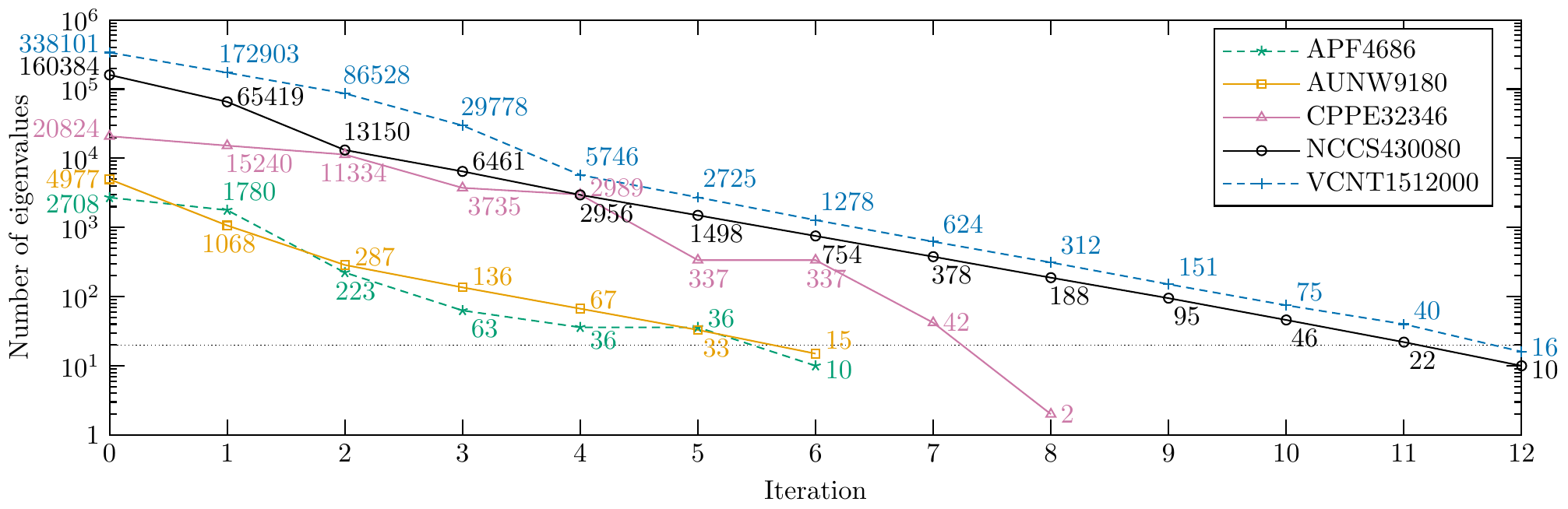}%
\caption{Number of eigenvalues in an interval after each bisection iteration}
\label{fig:bisection}
\end{figure}
\par
Here, we note that the straightforward bisection\footnote{In the straightforward bisection, line \ref{alg:bisect:modified:criterion} of Algorithm \ref{alg:bisect:modified} was changed to use the length of the interval as a stopping criterion.
Specifically, the stopping criterion was set to $(\sigma_{\mathrm{upper}}-\sigma_{\mathrm{lower}})/\max\{|\sigma_{\mathrm{lower}}|,|\sigma_{\mathrm{upper}}|\} < 10^{-14}.$
The $k$-th eigenvalue was computed using only bisection and from the equation $\hat{\lambda}_{k} = (\sigma_{\mathrm{lower}}+\sigma_{\mathrm{upper}})/2.$} required 47 to 49 iterations, which were roughly 4 to 8 times greater than those required for the three-stage algorithm shown in Figure \ref{fig:bisection} (6 to 12 iterations), to compute the $k$-th eigenvalue to the accuracy very similar to that in Table \ref{tab:kthev}%
{\interfootnotelinepenalty=10000.}
\par
%------------------------------------------------------------------------------------------%
\subsubsection{Computation of the $k$-th eigenpair}\label{sec:numerical:overview:eigenpair}%
%------------------------------------------------------------------------------------------%
In Table \ref{tab:iterationsi}, we show the iteration counts of Algorithm \ref{alg:validate} for computing $m$ eigenvalues in the interval $[\sigma_{\mathrm{lower}},\sigma_{\mathrm{upper}})$ of Table \ref{tab:intervalbi}.
The third column is the iteration count required for \eqref{eq:sigep:errbnd:cond:include} and \eqref{eq:sigep:errbnd:cond:disjoint} to be satisfied such that the index of each approximate eigenpair of the interval is validated.
The fourth and fifth columns represent the iteration counts required for the relative 2-norm of the residual and difference \eqref{eq:reldiff} of each approximate eigenpair of the interval to become less than $\tau_{\mathrm{res}} = 10^{-10}$ and $\tau_{\mathrm{diff}} = 10^{-10},$ respectively.
\par
\begin{table}[htbp]
\centering
\caption{Iteration counts of the SI Lanczos method}
\label{tab:iterationsi}
\begin{tabular}[c]{l|r|r|r|r}
\hline
\multicolumn{1}{c|}{\multirow{2}{*}{Data}} & \multicolumn{1}{c|}{\multirow{2}{*}{$m$}} & \multicolumn{3}{c}{Iteration}\\
\cline{3-5}
 & & Bound & Residual & Difference\\
\hline
APF4686 & 10 & 24 & 33 & 37\\
AUNW9180	& 15 & 26 & 42 & 48\\
CPPE32346	& 2 & 7 & 19 & 23\\
NCCS430080	 & 10 & 23 & 31 & 41\\
VCNT1512000	& 16 & 27 & 39 & 50\\
\hline
\end{tabular}
\end{table}
\par
Figure \ref{fig:error} shows the convergence history of the $k$-th eigenpair ($k=215040$) of NCCS430080.
$(\hat{\lambda}_{k}^{(j)},\hat{\boldsymbol{x}}_{k}^{(j)})$ in the figure legend denotes the $k$-th eigenpair computed at the $j$-th iteration of Algorithm \ref{alg:validate}.
$\boldsymbol{x}_{k}^{\mathrm{(d)}}$ represents the $k$-th eigenvector computed by the dense eigensolver.
As described in the legend, the figure shows the relative $2$-norm history of (\rom{1}) the residual, (\rom{2}) the difference between the $(j-1)$-th and $j$-th iterations defined in \eqref{eq:reldiff}, and (\rom{3}) the error compared with the dense eigensolver.
The figure also shows (\rom{4}) the pairwise $B$-orthogonality \eqref{eq:sigep:orthonorm} between the $k$-th eigenvector and the other ${m-1}$ eigenvectors of the interval in Table \ref{tab:intervalbi}.
Here, the three vertical dotted lines indicate the iteration counts of Bound, Residual, and Difference in Table \ref{tab:iterationsi}.
The horizontal dotted line indicates the convergence criteria $\tau_{\mathrm{res}}=10^{-10}$ and $\tau_{\mathrm{diff}}=10^{-10}$.
\par
As can be seen in Figure \ref{fig:error}, a small residual norm does not necessarily imply that eigenvector $\hat{\boldsymbol{x}}_{k}^{(j)}$ is close to convergence.
Indeed, the residual norm converged first, and convergence of the error and difference norms followed.
Since the error norm cannot be measured in general, the difference norm \eqref{eq:reldiff} is utilized in Algorithm \ref{alg:validate} to test for convergence.
\par
\begin{figure}[htbp]
\centering
\includegraphics[width=0.75\linewidth]{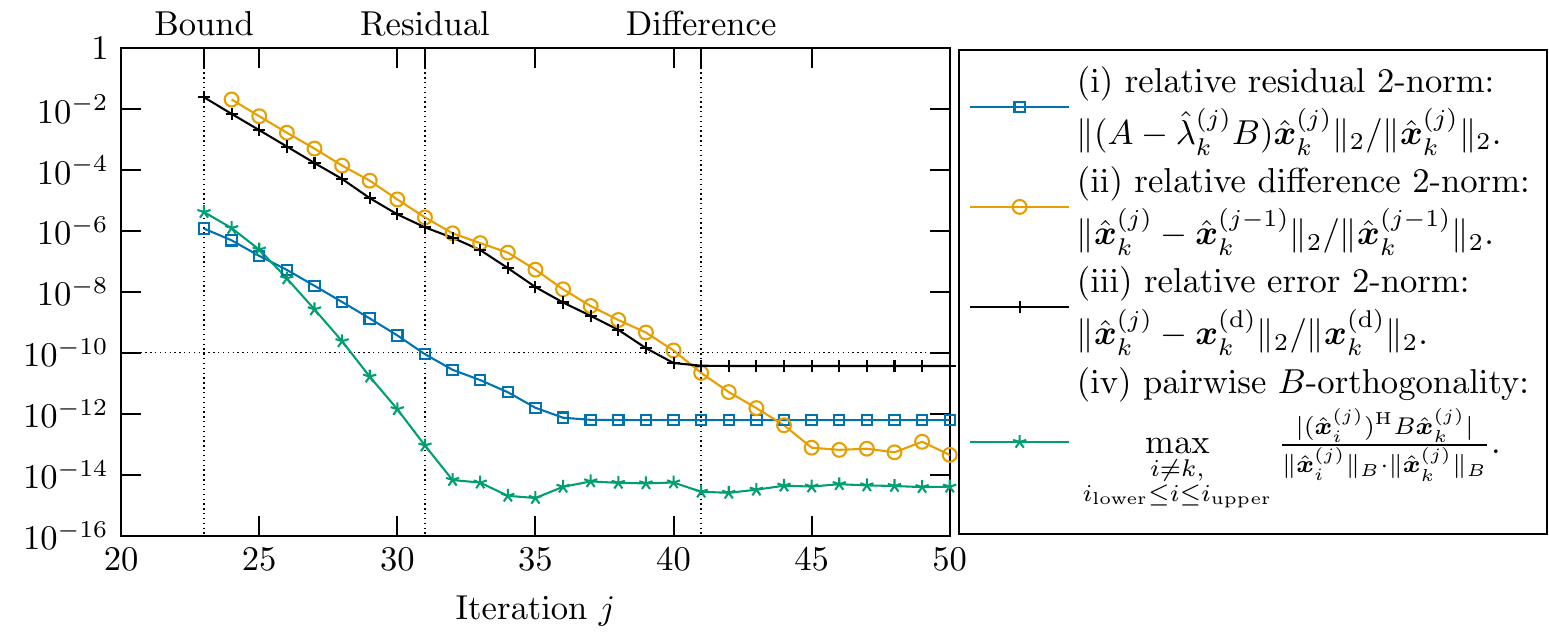}%
\caption{Convergence history of the $k$-th eigenpair ($k=215040$) of NCCS430080}
\label{fig:error}
\end{figure}
\par
%-------------------------------------------------%
\section{Concluding remarks}\label{sec:conclusion}%
%-------------------------------------------------%
The proposed three-stage algorithm obtained the validated $k$-th eigenpair $(\lambda_{k},\boldsymbol{x}_{k})$ for large sparse generalized Hermitian eigenvalue problems in electronic structure calculations with accuracy comparable with dense eigensolvers under limited computational resources.
The three-stage algorithm (Algorithm \ref{alg:frame}) consists of Algorithms \ref{alg:bisect:modified}, \ref{alg:initial}, and \ref{alg:validate}, each of which has been found to be effective for computation of the eigenpair and validation of its index.
In particular, from the numerical experiments, we have learned the following.
\begin{enumerate}
\item Algorithm \ref{alg:initial} can set a narrow interval containing $\lambda_{k}.$
The resulting intervals were 3 to 29 times narrower than $[\lambda_{1},\lambda_{n}],$ i.e., the tightest interval that can be obtained from some Gershgorin-type theorem in general.
In all experiments, only two iterations of Algorithm \ref{alg:initial} were required, i.e., the minimum required iterations to obtain an interval validated as containing $\lambda_{k}.$
\item Algorithm \ref{alg:validate} can compute the $k$-th eigenpair with high accuracy.
The eigenpairs computed by Algorithm \ref{alg:validate} agreed well with the results of dense eigensolvers, including a result obtained by a massively parallel eigenpair computation on the K computer.
Specifically, the eigenvalues were the same to at least 15 digits, and the relative error norms of the eigenvectors were less than $10^{-10}.$
\item By utilizing a sparse direct linear solver, large sparse matrices can be handled with efficiency.
For example, a nano-composite carbon solid problem of size $n=430080$ was solved in 0.6 hours using one core and 5.1 GB of memory on a workstation (a dense eigensolver required $2.9$ hours using 180000 cores and an estimated 5.9 TB of memory on the K computer).
\end{enumerate}%
Fortran codes for the three-stage algorithm are available on GitHub (\url{https://github.com/lee-djl/k-ep}).
\par
In future, we plan to examine parameter $m_{\mathrm{max}}.$
As explained in Section \ref{sec:app:over}, this parameter influences the overall performance of the three-stage algorithm because there is a trade-off between the second and third stages (Algorithms \ref{alg:bisect:modified} and \ref{alg:validate}).
In addition, we plan to compare spectral bisection with its variants.
As discussed in Section \ref{sec:bp:class:bisect}, it is possible to apply other root-finding algorithms to the second stage rather than bisection.
As long as the algorithms are derivative-free and a root is bracketed in the algorithms, they can be readily applied to the second stage and can locate $\lambda_{k}.$
Brent's method \cite{brent1971algorithm} is one such example.
Finally, we plan to modify the three-stage algorithm to deal with the presence of multiple eigenvalues or a cluster of eigenvalues.
Some possible modifications are described in Section \ref{sec:app:over}.
\par
%-------------------------%
\section*{Acknowledgments}%
%-------------------------%
The authors are very grateful to the reviewers for their careful reading of the manuscript and their valuable suggestions, which led to significant improvements in the presentation of the manuscript.
\par
This work was supported by JSPS KAKENHI (17H02828, 16KT0016, 16K17550) and by MEXT, Japan as a Priority Issue (Creation of new functional devices and high-performance materials to support next-generation industries) to be tackled using a post-K computer system.
This work used the computational resources of the K computer through the HPCI System Research Project (hp170147, hp170274), Oakforest-PACS through the JHPCN Project (jh170058-NAHI), and Sekirei provided by ISSP of the University of Tokyo.
\par
\appendix
%--------------------------------------%
\section*{Appendix}\label{sec:appendix}%
%--------------------------------------%
\def\thesubsection{\Alph{subsection}}
\numberwithin{equation}{subsection}
%--------------------------------------------------------------------------------------%
\subsection{Physical origin of the matrix eigenvalue problem}\label{sec:bp:phys:matrix}%
%--------------------------------------------------------------------------------------%
Practical electronic structure calculations use effective independent-electron theories, such as density functional theory \cite{HohenbergKohn1964,KohnSham1965,RevModPhys.71.1253}, that are formulated with an effective Schr\"odinger-type equation for electronic wave functions $\{ \psi_{i}(\boldsymbol{r}) \}_{i\ge1}$:
\begin{align*}
\mathcal{H}_{\mathrm{eff}} \psi_{i}(\boldsymbol{r}) = \varepsilon_{i} \psi_{i}(\boldsymbol{r})
\end{align*}
with the following Hamiltonian operator $\mathcal{H}_{\mathrm{eff}}$:
\begin{align*}
\mathcal{H}_{\mathrm{eff}} \equiv - \frac{\hbar^{2}}{2m_{\mathrm{e}}} \Delta + V_{\mathrm{eff}}(\boldsymbol{r}).
\end{align*}
Here, $m_{\mathrm{e}}$ is the electron mass, and $\hbar$ denotes the Planck constant.
The scalar function $V_{\mathrm{eff}}(\boldsymbol{r})$ is the potential function for electrons at the coordinate $\boldsymbol{r}$ and varies among materials.
The eigenvalues $\{ \varepsilon_{i} \}$ are real and can be indexed in increasing order.
\par
Independent-electron systems can be discretized by the Ritz variational method or the Galerkin method and can be formulated as a matrix eigenvalue problem.
When an electronic wave function $\psi_{i}(\boldsymbol{r})$ is expanded (more accurately, approximated) by the linear combination of $n$ non-orthogonal basis functions $\{ \chi_{j} (\boldsymbol{r}) \}_{j=1}^{n},$
\begin{align}\label{eq:bf}
\psi_{i}(\boldsymbol{r}) = \sum_{j=1}^{n} x_{j}^{(i)} \chi_{j}(\boldsymbol{r}),
\end{align}
generalized eigenvalue problem \eqref{eq:kep} appears with the $n \times n$ Hermitian matrices $A$ and $B$ whose $i,j$ elements are defined as follows:
\begin{align}\label{eq:mat_ele}
A_{ij} \equiv \int \chi_{i}^{\ast}(\boldsymbol{r}) \mathcal{H}_{\mathrm{eff}} \chi_{j}(\boldsymbol{r}) d\boldsymbol{r},\quad
B_{ij} \equiv \int \chi_{i}^{\ast}(\boldsymbol{r}) \chi_{j}(\boldsymbol{r}) d\boldsymbol{r}.
\end{align}
Here, the asterisk $({}^{\ast})$ denotes the complex conjugate of a function.
Matrix $B$ is positive definite from the definition, and its diagonal elements all equal one provided that the basis functions are normalized.
\par%
The size and structure of these matrices depend on the construction of the Hamiltonian operator $\mathcal{H}_{\mathrm{eff}}$ and the choice of the basis set $\{ \chi_{j} (\boldsymbol{r}) \}.$
This paper is based on a first-principle-based modeled (transferable tight-binding) theory \cite{hoshi2012order} in which basis functions, referred to as atomic orbitals, are localized in real space with their localization center being the position of an atom in a material.
The index $j$ of the basis functions $\chi_{j}$ can then be expressed as a composite of two indices $l$ and $m$ that represent a localization center, or an atom, and the shape of an orbital, respectively.
Using indices $l$ and $m,$ we obtain an alternative expression for the expansion of wave functions \eqref{eq:bf}:
\begin{align*}
\psi_{i}(\boldsymbol{r}) = \sum_{l=1}^{n_{\mathrm{atom}}} \sum_{m=1}^{n_{l}} x_{lm}^{(i)} \chi_{lm}(\boldsymbol{r}).
\end{align*}
Here, $n_{\mathrm{atom}}$ denotes the number of atoms in a material.
$n_{l}$ is the number of orbitals centered at the atom $l$ and differs depending on the given atomic species.
Therefore, the matrix size $n$ is roughly proportional to the number of atoms $n_{\mathrm{atom}}$ in a material.
Since orbitals are localized in real space, the matrix elements \eqref{eq:mat_ele} decay rapidly as the distance between the localization centers of the orbitals increases.
Thus, the matrices become sparse.
Further details about the physical origin of the problem can be found in the literature \cite{hoshi2016extremely}.
\par
%---------------------------------------------------------------------------------------%
\subsection{Physical background of the $k$-th eigenvalue problem}\label{sec:bp:phys:kth}%
%---------------------------------------------------------------------------------------%
The $k$-th eigenpair is associated with the HO state, which is in close relationship with several material properties, such as electronic transport and optical spectra \cite[Chapter 2]{Martin2004}.
The target index $k$ is a material-specific value that is uniquely determined by the number of electrons $n_{\mathrm{elec}}$ in a material.
In para-spin materials calculations, which represent a typical case, the index is defined as one-half the number of electrons, or $k \equiv \lceil n_{\rm elec}/2 \rceil.$
The difference between the $k$-th and $(k+1)$-th eigenvalues, or $\varepsilon_{k+1} - \varepsilon_{k}$, is referred to as the energy gap, which is crucial for electronic properties because the value is zero in metallic materials and non-zero in semiconducting or insulating materials.
Therefore, the $k$-th and $(k+1)$-th eigenvalues should be rigorously distinguished.
\par
One of our recent motivations to address the $k$-th eigenvalue problem is electronic transport calculations of organic device materials using a quantum wave (wave packet) dynamics method \cite{imachi2016one,hoshi2016extremely}.
In this method, an exited electronic wave (of an excited electron or a hole) is simulated by real-time dynamics with an effective time-dependent Schr\"odinger-type equation, and the wave function of the HO state or states near the HO state is set as the initial state of the wave.
\par
%--------------------------%
\bibliographystyle{unsrt}

%
%--------------------------%
\end{document}